\documentclass[10pt,reqno]{amsart}

\usepackage{amsthm, mathrsfs, amsmath, amstext, esint, amsxtra, amsfonts, slashed, dsfont, amssymb}
\usepackage{relsize}
\usepackage{enumitem}
\usepackage[dvipsnames]{xcolor}
\usepackage[colorlinks, linkcolor=blue, citecolor=red, urlcolor=blue, pagebackref, hypertexnames=false]{hyperref}

\usepackage{lmodern}
\usepackage[T1]{fontenc}

\usepackage[belowskip=2pt,aboveskip=0pt]{caption}
\usepackage{booktabs}
\usepackage{geometry}
\newcommand{\Rb}{\mathbb R} 
\newcommand{\Cb}{\mathbb C} 
\newcommand{\Eb}{\mathbb E} 
\newcommand{\Mb}{\mathbb M}
\newcommand{\Gb}{\mathbb G}
\newcommand{\Kb}{\mathbb K}
\newcommand{\Hb}{\mathbb H}
\newcommand{\Pb}{\mathbb P}  
\newcommand{\Wb}{\mathbb W} 
\newcommand{\Ac}{\mathcal A}

\newcommand{\Hc}{\mathcal H}
\newcommand{\Lc}{\mathcal L}

\newcommand{\Fc}{\mathcal F}
\newcommand{\Vc}{\mathcal V}

\newcommand{\Wc}{\mathcal W}
\newcommand{\Mcal}{\mathcal M}
  
\newcommand{\Sc}{\mathcal S}
\newcommand{\vareps}{\varepsilon}

\newcommand{\ic}{{\rm i}}
\newcommand{\dd}{\;{\rm d}}

\DeclareMathOperator*{\loc}{loc}

\DeclareMathOperator*{\supp}{supp}

\DeclareMathOperator*{\ima}{\Im}
\DeclareMathOperator*{\rea}{\Re}

\DeclareMathOperator*{\GN}{GN}
\DeclareMathOperator*{\gn}{gn}

\newcommand{\ee}{{\rm e}}

\numberwithin{equation}{section}
\newtheorem{theorem}{Theorem}[section]
\newtheorem{lemma}[theorem]{Lemma}
\newtheorem{proposition}[theorem]{Proposition}

\newtheorem{definition}[theorem]{Definition}

\newtheoremstyle{remarkstyle}
{}{}{
}{ }{\bfseries}{.}{ }{\thmname{#1}\thmnumber{ #2}\thmnote{ (#3)}}
\theoremstyle{remarkstyle}
\newtheorem{remark}{Remark}[section]

\usepackage[depth=3]{bookmark}

\title[INLS System]{A system of inhomogeneous NLS arising in optical media with a $\chi^{(2)}$ nonlinearity, part I : Dynamics} 

\author[V. D. Dinh \& A. Esfahani]{Van Duong Dinh$^\dagger$ and Amin Esfahani$^*$}

\subjclass[2020]{35Q55; 35B44; 35P25\\ 
	$^\dagger$Ecole Normale Sup\'erieure de Lyon \& CNRS, UMPA (UMR 5669), France, contact@duongdinh.com\\
	$^*$Corresponding Author. Department of Mathematics, Nazarbayev University, Astana 010000, Kazakhstan, saesfahani@gmail.com, amin.esfahani@nu.edu.kz
	}

\begin{document}

\begin{abstract}
We study a system of inhomogeneous nonlinear Schrödinger equations that emerge in optical media with a $\chi^{(2)}$ nonlinearity. This nonlinearity, whose local strength is subject to a cusp-shaped spatial modulation, $\chi^{(2)}\sim |x|^{-\alpha}$ where $\alpha > 0$, can be induced by spatially non-uniform poling. Our first step is to establish a vectorial Gagliardo--Nirenberg type inequality related to the system. This allows us to identify the necessary conditions on the initial data that lead to the existence of global in time solutions. By exploiting the spatial decay at infinity of the nonlinearity, we demonstrate the non-radial energy scattering in the mass-supercritical regime for global solutions. These solutions have initial data that lie below a mass-energy threshold, regardless of whether the system is mass-resonant or non-mass resonant. Lastly, we provide the criteria for the existence of non-radial blow-up solutions with mass-critical and mass-supercritical nonlinearities in both mass and non-mass resonance cases.
\end{abstract}

\maketitle


\section{Introduction}
\label{S-intro}
\setcounter{equation}{0}

The Nonlinear Schrödinger Equation (NLS) plays a crucial role in the field of nonlinear optics, particularly in modeling the propagation of intense laser beams within a homogeneous bulk medium exhibiting Kerr nonlinearity. It's a well-established fact that the NLS, which governs beam propagation, cannot support stable high-power propagation within a homogeneous bulk medium. However, towards the end of the previous century, it was proposed that stable high-power propagation could be realized in plasma. This could be achieved by transmitting an initial laser beam that forms a channel with diminished electron density, thereby reducing the nonlinearity within the channel (refer to \cite{Gill, LT} for more details).

Under such circumstances, the beam propagation can be represented by the inhomogeneous nonlinear Schrödinger equation, given by:
\begin{align} \label{INLS-K}
\ic \partial_t  u + \Delta u + K(x) |u|^p u =0, \quad (t,x) \in \Rb \times \Rb^d,
\end{align}
In this equation, $u$ represents the electric field in laser and optics, $\alpha >0$ denotes the power of nonlinear interaction, and the potential $K(x)$ is proportional to the electron density. Towers and Malomed \cite{TM}, through variational approximation and direct simulations, observed that for a specific type of nonlinear medium, the equation \eqref{INLS-K} results in completely stable beams.

When $K(x)$ is a constant, \eqref{INLS-K} is the standard NLS equation which is a subject of thorough study over the past several decades (for more details, refer to the monographs \cite{Cazenave, SS, Tao}).

When $K(x)$ is a non-constant bounded function, Merle \cite{Merle} proved the existence and nonexistence of minimal blow-up solutions with $p=\frac{4}{d}$. Rapha\"el and Szeftel \cite{RS} established sufficient conditions for the existence, uniqueness and characterization of minimial blow-up solutions to the equation.

When $K(x)$ behaves like $|x|^{-\alpha}$ with $\alpha>0$, equation \eqref{INLS-K} has been extensively studied over the past several years. Genoud and Stuart \cite{GS} demonstrated local well-posedness in $H^1(\mathbb{R}^d)$. Farah \cite{Farah} established the global existence of $H^1$-solutions and proved the existence of finite time blow-up solutions for initial data with finite variance. Dinh \cite{Dinh-NA} extended the existence of blow-up solutions to radial data, while Cardoso and Farah \cite{CF}, and Bai and Li \cite{BL} did so for non-radial initial data. Energy scattering was initially proven for radial initial data by Farah and Guzmán \cite{FG-1, FG-2}, Campos \cite{Campos}, and Dinh \cite{Dinh-JHDE}, and was later extended to non-radial initial data by Miao, Murphy, and Zheng \cite{MMZ}, as well as Cardoso, Farah, Guzmán, and Murphy \cite{CFGM} (see also \cite{DK, Murphy, CC-PAMS}).

In this paper, we consider a system of inhomogeneous nonlinear Schr\"odinger equations in optical media, namely
\begin{align} \label{INLS}
\left\{
\renewcommand*{\arraystretch}{1.3}
\begin{array}{l}
\ic \partial_t u + \frac{1}{2}\Delta u +|x|^{-\alpha} \overline{u} v =0,\\
\ic \partial_t v + \frac{\kappa}{2} \Delta v - \gamma v +\frac{1}{2} |x|^{-\alpha} u^2 =0, 
\end{array}
\right.
\end{align}
where $u, v: \Rb_t \times \Rb^d_x \rightarrow \Cb$ are unknown wave functions, $d\geq 1$, $\alpha>0$, $\kappa>0$, and $\gamma \in \Rb$. This system describes the two-wave (degenerate, alias type-I) quadratic interactions for the complex fundamental-frequency (FF) and second-harmonic (SH) amplitudes in the presence of the spatial singular modulation of the $\chi^{(2)}$ nonlinearity. The real coefficient $\gamma$ represents the SH-FF mismatch and, by rescaling, can be $0, \pm 1$. We refer the readers to \cite{LM-PRA, LM-EPJST} for a derivation of this system from the physical context.

\vspace{2mm}

In the nonsingular case $\alpha=0$, where \eqref{INLS} was introduced as a non-relativistic version of some Klein-Gordon systems (see \cite{HOT}), the well-posedness in $L^2(\Rb^d)\times L^2(\Rb^d)$ with $1\leq d\leq 4$ and in $H^1(\Rb^d)\times H^1(\Rb^d)$ with $1\leq d\leq 6$ were completely investigated in \cite{HOT} by the contraction argument combined with Strichartz estimates. Moreover, the global well-posedness in $L^2(\Rb^d)\times L^2(\Rb^d)$ when $d = 4$ as well as the global existence in $H^1(\Rb^d)\times H^1(\Rb^d)$ when $1 \leq d \leq 3$ were obtained due to the conservation of mass and energy. In addition, for initial data in $H^1(\Rb^5)\times H^1(\Rb^5)$, a condition for global existence was found which depends on the size of the initial data when compared to the associated ground states (\cite{HOT}). The global existence and energy scattering were studied in \cite{Hamano, HIN, WY, MX}, while the finite time blow-up were showed in \cite{Dinh, NP, IKN}. See also \cite{NP-CCM, NP-PDEA, NP-CVPDE, DF} for other results related to more general systems of NLS with quadratic interaction.  

\vspace{2mm}

By applying the abstract argument due to Cazenave \cite{Cazenave}, we show that system \eqref{INLS} is locally well-posed in $\Hc^1 := H^1(\Rb^d) \times H^1(\Rb^d)$. More precisely, we have the following local wellposedness result.

\begin{proposition}[Local wellposedness] \label{prop-lwp-inls}
	Let $1\leq d\leq 5$, $\kappa>0$, $\gamma \in \Rb$, $0<\alpha <\min\left\{2,d\right\}$, and $\alpha<\frac{6-d}{2}$ if $3\leq d\leq 5$. For any $(u_0,v_0)\in \Hc^1$, there exists a unique maximal solution 
	\[
	(u,v)\in C((-T_*,T^*), \Hc^1) \cap C^1((-T_*,T^*), \Hc^{-1})
	\]
	to \eqref{INLS} with initial data $\left.(u,v)\right|_{t=0}=(u_0,v_0)$, where $\Hc^{-1}:= H^{-1}(\Rb^d) \times H^{-1}(\Rb^d)$ is the dual space of $\Hc^1$.  The maximal time satisfies the blow-up alternative: if $T^*<\infty$ (resp. $T_*<\infty$), then 
	\[
	\lim_{t\nearrow T^*}\|(u(t),v(t))\|_{\Hc^1}=\infty \quad \left( \text{resp. } \lim_{t\searrow -T_*}\|(u(t),v(t))\|_{\Hc^1}=\infty \right).
	\] 
	In addition, there are conservation laws of mass and energy, i.e.,
	\begin{align*}
		\Mb(u(t),v(t)) &= \|u(t)\|^2_{L^2} + 2\|v(t)\|^2_{L^2} = \Mb(u_0,v_0), \\
		\Eb(u(t),v(t)) &= \frac{1}{2} \|\nabla u(t)\|^2_{L^2} + \frac{\kappa}{2}\|\nabla v(t)\|^2_{L^2} + \gamma\|v(t)\|^2_{L^2} - \rea \int_{\Rb^d} |x|^{-\alpha} u^2(t) \overline{v}(t) \dd x = \Eb(u_0,v_0), 
	\end{align*}
	for all $t\in (-T_*,T^*)$.
\end{proposition}

When $\gamma=0$, we notice that system \eqref{INLS} has a scaling invariance
\[
(u_\lambda(t,x), v_\lambda(t,x)) = \left(\lambda^{2-\alpha} u(\lambda^2 t, \lambda x), \lambda^{2-\alpha} v(\lambda^2 t, \lambda x) \right), \quad \lambda>0.
\]
That is, if $(u(t),v(t))$ is a solution to \eqref{INLS} with initial data $(u_0,v_0)$, then $(u_\lambda(t),v_\lambda(t))$ is also a solution to \eqref{INLS} with initial data $(u_\lambda(0), v_\lambda(0))$. This scaling leaves the $\dot{\Hc}^{s_c}$-norm of initial data invariant, i.e.,
\[
\|(u_\lambda(0), v_\lambda(0))\|_{\dot{\Hc}^{s_c}}:=\|u_\lambda(0)\|_{\dot{H}^{s_c}} + \|v_\lambda(0)\|_{\dot{H}^{s_c}} 
= \|(u_0,v_0)\|_{\dot{\Hc}^{s_c}}, \quad \forall \lambda>0,
\] 
where
\begin{align} \label{sc}
s_c:=\frac{d}{2}- 2+\alpha
\end{align}
is the critical Sobolev exponent. When $s_c=0$ or $\alpha=\frac{4-d}{2}$, system \eqref{INLS} is called mass-critical; while for $s_c=1$ or $\alpha=\frac{6-d}{2}$, system \eqref{INLS} is known as energy-critical. According to this terminology, our local well-posedness result is available only for the energy-subcritical regime.
The case of energy-critical nonlinearity $\alpha=\frac{6-d}{2}$ with $3\leq d\leq 5$ is special and will be addressed in a forthcoming work. 

\vspace{2mm}

The main purpose of this paper is to give necessary and sufficient conditions on initial data which clarify the
global existence, energy scattering, and blowing-up behavior of solutions to \eqref{INLS}. To obtain these results, we 
show, in Section \ref{S-ineq}, the following vectorial Gagliardo--Nirenberg type inequality
\[
\Pb(u,v)\leq C_{\GN} \left(\Kb(u,v)\right)^{\frac{d+2\alpha}{4}} \left(\Mb(u,v)\right)^{\frac{6-d-2\alpha}{4}}, \quad \forall (u,v)\in \Hc^1,
\]
where
\begin{align*}
\Kb(u,v) &:= \|\nabla u\|^2_{L^2}+ \kappa \|\nabla v\|^2_{L^2}, \\
\Pb(u,v) &:= \Re \int_{\Rb^d} |x|^{-\alpha} u^2 \overline{v} \dd x.
\end{align*}
The proof of \eqref{GN-ineq} follows from a standard argument of Weinstein \cite{Weinstein} (see also \cite{Farah}). The only difference here is that instead of using the Schwarz symmetrization to exploit the radical compact embedding, we make use of the decay factor $|x|^{-\alpha}$ to show a compactness result (see Lemma \ref{lem-comp}) in the non-radial setting. The sharp constant $C_{\GN}$ is attained by a non-trivial solution to 
\begin{align} \label{equ-var-psi}
	\left\{
	\renewcommand*{\arraystretch}{1.3}
	\begin{array}{l}
		\frac{1}{2}\Delta \varphi -\varphi +|x|^{-\alpha} \overline{\varphi} \psi =0,\\
		\frac{\kappa}{2} \Delta \psi-2\psi +\frac{1}{2} |x|^{-\alpha} \varphi^2 =0,
	\end{array}
	\right.
	\quad x\in \Rb^d.
\end{align}
\begin{definition} [Nonlinear ground state] \label{def-non-gs}
	A non-trivial solution to \eqref{equ-var-psi} is called a nonlinear ground state if it optimizes the Gagliardo--Nirenberg type inequality \eqref{GN-ineq}.  
\end{definition}

As a direct consequence of the vectorial Gagliardo--Nirenberg inequality, we provide sufficient conditions on the initial data under which the solution of \eqref{INLS} exists globally in time. 

\begin{proposition}[Global existence] \label{prop-gwp}
	Let $d, \kappa, \gamma, \alpha$ be as in Proposition \ref{prop-lwp-inls} and $(u_0,v_0) \in \mathcal H^1$. Let $(\varphi,\psi)$ be a nonlinear ground state. Assume that one of the following conditions is satisfied:
	\begin{itemize}
		\item[(1)] (Mass-subcritical) $1\leq d\leq 3$ and $0<\alpha<\frac{4-d}{2}$;
		\item[(2)] (Mass-critical) $2\leq d\leq 3$, $\alpha=\frac{4-d}{2}$, and
		\begin{align} \label{cond-gwp-crit}
		\mathbb M(u_0,v_0) < \mathbb M(\varphi,\psi);
		\end{align}
		\item[(3)] (Mass-supercritical) $2\leq d\leq 5$, $\frac{4-d}{2}<\alpha<\frac{6-d}{2}$, and
		\begin{align} 
			\Hb(u_0,v_0) \left(\Mb(u_0,v_0)\right)^\sigma &< \Eb_0(\varphi,\psi) \left(\Mb(\varphi,\psi)\right)^\sigma, \label{cond-gwp-supe-1}\\
			\Kb(u_0,v_0) \left(\Mb(u_0,v_0)\right)^\sigma &<\Kb(\varphi,\psi) \left(\Mb(\varphi,\psi)\right)^\sigma, \label{cond-gwp-supe-2}
		\end{align}
		where $\sigma = \frac{6-d-2\alpha}{d+2\alpha-4}$, 
		\begin{align} \label{Hb}
			\Hb(u,v) := \left\{
			\begin{array}{ccl}
				\Eb(u,v) &\text{if}& \gamma \geq 0, \\
				\Eb(u,v) +\frac{|\gamma|}{2} \Mb(u,v) &\text{if}& \gamma<0.
			\end{array}
			\right.
		\end{align}
		and
		\begin{align} \label{Eb0}
			\Eb_0(\varphi,\psi):= \frac{1}{2} \Kb(\varphi,\psi) - \Pb(\varphi,\psi).
		\end{align}
	\end{itemize}
	Then the corresponding solution to \eqref{INLS} exists globally in time. In addition, we have
	$$
	\sup_{t\in \mathbb R} \|(u(t),v(t))\|_{\mathcal H^1} \leq C(u_0,v_0,\varphi,\psi).
	$$
\end{proposition}

\begin{remark}
	In the mass-critical case, the mass of nonlinear ground state remains unaffected by the selection of nonlinear ground states. Similarly, in the mass-supercritical regime, the quantities on the right hand side of \eqref{cond-gwp-supe-1} and \eqref{cond-gwp-supe-2} are independent of the choice of nonlinear ground states.
\end{remark}


Once global solutions exist, it is natural to study the long time behavior of these solutions. We have the following energy scattering for global solutions in the mass-supercritical regime.

\begin{theorem}[Energy scattering] \label{theo-scat}
	Let $3\leq d\leq 5$, $0<\alpha<\min\left\{2,\frac{d}{2}\right\}$, $\frac{4-d}{2}<\alpha<\frac{6-d}{2}$, $\kappa>0$, and $\gamma \in \Rb$. Let $(u_0,v_0) \in \Hc^1$ satisfy \eqref{cond-gwp-supe-1} and \eqref{cond-gwp-supe-2}. Then the corresponding solution to \eqref{INLS} scatters in $\Hc^1$ in both directions in the sense that there exist $(u_\pm,v_\pm)\in \Hc^1$ such that
	\[
	\lim_{t\rightarrow \pm \infty} \|(u(t),v(t))-(\Sc_1(t) u_\pm, \Sc_2(t) v_\pm)\|_{\Hc^1}=0,
	\]
	where $(\Sc_1(t), \Sc_2(t)) = \left(e^{it\frac{1}{2} \Delta}, e^{it \left(\frac{\kappa}{2} \Delta -\gamma\right)}\right)$ is the linear propagator.
\end{theorem}

\begin{remark}
	The restriction $\alpha<\frac{d}{2}$ is technical due to our nonlinear estimates (see Lemma \ref{lem-non-est-1}). This prevents us to show the energy scattering in two dimensions. 
\end{remark}

The main interest of Theorem \ref{theo-scat} is the energy scattering for general (non-radial) initial data, irrespective of whether it falls within mass-resonance or non-mass-resonance cases, provided it remains below a certain mass-energy threshold. To our knowledge, the energy scattering for quadratic system of NLS (i.e., $\alpha=0$ and $\gamma=0$) with non-radial data was known only for a slight perturbation of the mass resonance case, i.e., $|\kappa-1/2|\ll 1$ (see e.g., \cite{Hamano, HIN, MX, WY, NP-PDEA}). Here we prove the non-radial scattering for all $\kappa>0$. Our strategy is based on a recent work of Murphy \cite{Murphy} which is a combination of a scattering criterion and a space-time estimate. Here we make an effective use of the spatial decay at infinity of $|x|^{-\alpha}$. 

We next turn our attention to finding conditions under which the solutions of \eqref{INLS} blow up or grow up. To this end, we introduce
\begin{align} \label{Gb}
\Gb(u,v) = \Kb(u,v) -\frac{d+2\alpha}{2} \Pb(u,v)
\end{align}
which is nothing but the Pohozaev functional obtained by 
\[
\frac{\dd}{\dd t} \left( 2 \ima \int_{\Rb^d} x \cdot (\nabla u(t) \overline{u}(t) + \nabla v(t) \overline{v}(t)) \dd x\right) = 2\Gb(u(t),v(t)).
\]
We note that the left hand side is exactly the time derivative of the standard virial quantity 
\[
\mathcal V(t) = \int_{\Rb^d} |x|^2\left(|u(t)|^2+2|v(t)|^2\right) \dd x
\]
only if $\kappa=\frac{1}{2}$, which is usually referred to as the mass-resonance condition. We have the following blow-up or grow-up result in the mass-critical case.

\begin{theorem}[Blow-up or grow-up in the mass-critical case] \label{theo-blow-crit}
	Let $2\leq d\leq 3$, $\alpha=\frac{4-d}{2}$, $\kappa>0$, and $\gamma \in \Rb$. Let $(u_0,v_0) \in \Hc^1$ and $(u,v) \in C((-T_*,T^*), \Hc^1)$ be the corresponding maximal solution to \eqref{INLS}. Assume that 
	\begin{align} \label{blow-cond}
		\sup_{t\in(-T_*,T^*)} \Gb(u(t),v(t)) \leq -\delta
	\end{align}
	for some constant $\delta>0$, where $\Gb(u,v)$ is as in \eqref{Gb}.
	\begin{itemize}
		\item If $\kappa=\frac{1}{2}$, then the solution blows up in finite time, i.e., $T_*,T^*<\infty$ and 
		\[
		\lim_{t\nearrow T^*} \|(u(t),v(t))\|_{\Hc^1} =\infty, \quad \lim_{t\searrow -T_*} \|(u(t),v(t))\|_{\Hc^1} = \infty.
		\]
		\item If $\kappa \ne \frac{1}{2}$, then the solution either blows up in finite time or it blows up in infinite time in the sense that $T^*=\infty$ and there exists $C>0$ such that
		\[
		\Kb(u(t),v(t)) \geq Ct^2, \quad \forall t\geq t_0
		\]
		for some $t_0>0$ sufficiently large. A similar result holds for negative times.
	\end{itemize} 
	In particular, if $\Hb(u_0,v_0)<0$, where $\Hb$ is as in \eqref{Hb}, then the above blow-up/grow-up results hold.
\end{theorem}

\begin{theorem}[Blow-up or grow-up in the mass-supercritical case] \label{theo-blow-supe}
	Let $2\leq d\leq 5$, $0<\alpha<\min\{2,d\}$, $\alpha>\frac{4-d}{2}$, $\alpha<\frac{6-d}{2}$ if $3\leq d\leq 5$, $\kappa>0$, and $\gamma \in \Rb$. Let $(u_0,v_0) \in \Hc^1$ and $(u,v) \in C((-T_*,T^*), \Hc^1)$ be the corresponding maximal solution to \eqref{INLS}. Assume that \eqref{blow-cond} holds.
	\begin{itemize}
		\item If $2\leq d\leq 4$, then the solution blows up in finite time.
		\item If $d=5$, then the solution either blows up in finite time or it blows up in infinite time in the sense that $T^*=\infty$ and 
		\begin{align} \label{blow-infi-supe}
			\sup_{t\in [0,\infty)} \Kb(u(t),v(t)) =\infty.
		\end{align}
		Moreover, we have for all $T>0$,
		\begin{align} \label{blow-rate-T}
			\sup_{t\in [0,T]} \Kb(u(t),v(t)) \geq \left\{
			\begin{array}{lcl}
				C T^4 &\text{if}& \kappa=\frac{1}{2}, \\
				C T^{4/3} &\text{if}& \kappa \ne \frac{1}{2}.
			\end{array}
			\right.
		\end{align}
		A similar statement holds for negative times.
	\end{itemize} 
	In particular, if $(u_0,v_0) \in \Hc^1$ satisfies either $\Hb(u_0,v_0)<0$ or, if $\Hb(u_0,v_0) \geq 0$, we assume that
	\begin{align} 
		\Hb(u_0,v_0) \left(\Mb(u_0,v_0)\right)^\sigma &<\Eb_0(\varphi,\psi) \left(\Mb(\varphi,\psi)\right)^\sigma, \label{cond-blow-supe-1} \\
		\Kb(u_0,v_0) \left(\Mb(u_0,v_0)\right)^\sigma &>\Kb(\varphi,\psi) \left(\Mb(\varphi,\psi)\right)^\sigma, \label{cond-blow-supe-2}
	\end{align} 
	where $(\varphi,\psi)$ is a nonlinear ground state and $\Hb$ is as in \eqref{Hb}, then the above blow-up/grow-up results hold.
\end{theorem}

Our blow-up or grow-up results do not assume any symmetric condition or finite invariance of initial data and they hold regardless the mass-resonance or non-mass resonance cases (i.e., for all $\kappa>0$). These extend recent results on the blow-up/grow-up of the single inhomogeneous nonlinear Schr\"odinger equation (see \cite{BL,CF}). Comparing to the single inhomogeneous NLS, system \eqref{INLS} has an additional difficulty appeared in the non mass-resonance case $\kappa \ne \frac{1}{2}$, namely 
$$
\frac{d^2}{dt^2} \mathcal V(t) \ne 2 \Gb(u(t),v(t)).
$$
This prevents us from using the convexity argument as in \cite{BL, CF}. The proofs of Theorems \ref{theo-blow-crit} and \ref{theo-blow-supe} are based on localized virial estimates which was proposed by Ogawa and Tsutsumi \cite{OT-1, OT-2}. Here instead of making use of the radial assumption, we exploit the spatial decay at infinity of the nonlinearity. In the non-mass resonance case, we exploit an ODE technique which is inspired by Boulenger, Himmelsbach and Lenzmann's work \cite{BHL}. 

We stress that blow-up/grow-up results for the quadratic system of NLS (i.e. $\alpha=\gamma=0$) in the non-mass resonance case $\kappa \ne 1/2$ is only available under the radial symmetry or finite variance condition (see \cite{Dinh-NA, NP, IKN}) and the cylindrical symmetry assumption (see \cite{DF}). Therefore, we obtain stronger results comparing to the ones for system of NLS with quadratic interaction.

\vspace{2mm}

We finish the introduction by listing some notations which will be used throughout the sequel.

\paragraph{\bf Notations}
Denote
\begin{align*}
\Hc^1:&= H^1(\Rb^d) \times H^1(\Rb^d), \quad \dot{\Hc}^1:= \dot{H}^1(\Rb^d) \times \dot{H}^1(\Rb^d), \\
\Lc^p :&= L^p(\Rb^d) \times L^p(\Rb^d), \quad \Wc^{1,p}:= W^{1,p}(\Rb^d)\times W^{1,p}(\Rb^d)
\end{align*}
with the norms
\begin{align*}
\|(u,v)\|_{\Hc^1} :&= \|u\|_{H^1} + \|v\|_{H^1},\quad \|(u,v)\|_{\dot{\Hc}^1} := \|u\|_{\dot{H}^1} + \|v\|_{\dot{H}^1}, \\
\|(u,v)\|_{\Lc^p}:&=\|u\|_{L^p} +\|v\|_{L^p}, \quad \|(u,v)\|_{\Wc^{1,p}} := \|u\|_{W^{1,p}} + \|v\|_{W^{1,p}}.
\end{align*}
When $p=2$, we use $\Hc^1$ instead of $\Wc^{1,2}$. Let $I\subset \Rb$ be an interval. We denote
\[
\Lc^p(I, \Lc^q):= L^p(I, L^q(\Rb^d)) \times L^p(I, L^q(\Rb^d)), \quad \Lc^p(I, \Wc^{1,q}) := L^p(I, W^{1,q}(\Rb^d)) \times L^p(I, W^{1,q}(\Rb^d))
\]
with norms
\[
\|(u,v)\|_{\Lc^p(I, \Lc^q)} := \|u\|_{L^p(I,L^q)} + \|v\|_{L^p(I,L^q)}, \quad \|(u,v)\|_{\Lc^p(I,\Wc^{1,q})}:= \|u\|_{L^p(I, W^{1,q})} +\|v\|_{L^p(I, W^{1,q})}.
\]

\section{A vectorial Gagliardo-Nirenberg inequality}
\label{S-ineq}

\begin{proposition}[Gagliardo--Nirenberg inequality] \label{prop-GN-ineq}
	Let $1\leq d \leq 5$, $\kappa>0$, $0<\alpha<\min\{2,d\}$ and $\alpha<\frac{6-d}{2}$ if $3\leq d\leq 5$. Then the following inequality holds for all $(u,v)\in \Hc^1$:
	\begin{align}\label{GN-ineq}
	\Pb(u,v)\leq C_{\GN} \left(\Kb(u,v)\right)^{\frac{d+2\alpha}{4}} \left(\Mb(u,v)\right)^{\frac{6-d-2\alpha}{4}}.
	\end{align}
	The sharp constant $C_{\GN}$ is attained by a pair of functions $(\varphi,\psi)$ which is a non-trivial solution to \eqref{equ-var-psi}. In particular, we can take $(\varphi,\psi)$ to be positive, radially symmetric, and radially decreasing.
\end{proposition}

Before giving the proof of Proposition \ref{prop-GN-ineq}, let us recall the corresponding scalar weighted Gagliardo--Nirenberg inequality. 
\begin{proposition} [\cite{CKN,Farah, Yanagida}]\label{prop-gn-sob-ineq} Let $1\leq d\leq 5$, $0<\alpha<\min\{2,d\}$, and $\alpha<\frac{6-d}{2}$ if $3\leq d \leq 5$. Then we have
		\begin{align} \label{gn-ineq}
		\int_{\Rb^d} |x|^{-\alpha} |f(x)|^3 \dd x \leq C_{\gn} \|\nabla f\|^{\frac{d+2\alpha}{2}}_{L^2} \|f\|^{\frac{6-d-2\alpha}{2}}_{L^2}, \quad f\in H^1(\Rb^d).
		\end{align}
		The optimal constant $C_{\gn}$ is attained by a unique positive radial solution to 
		\begin{align} \label{Q}
		\Delta Q-Q+|\alpha|^{-\alpha} Q^2=0, \quad x \in \Rb^d.
		\end{align}
\end{proposition}

\begin{remark}
	The inequality \eqref{gn-ineq} (without optimal constants) is a special case of the so-called Cafferelli-Kohn-Nirenberg inequalities \cite{CKN}. The sharp constant in \eqref{gn-ineq} was proved in \cite{Farah}. The uniqueness of positive radial solutions to \eqref{Q} is due to \cite{Yanagida}. 
\end{remark}

To prove Proposition \ref{prop-GN-ineq}, we need the following compactness result which makes use of the spatial decay at infinity of the weighted term $|x|^{-\alpha}$.
\begin{lemma}\label{lem-comp}
	Let $1\leq d\leq 5$, $0<\alpha<\min\{2,d\}$, and $\alpha<\frac{6-d}{2}$ if $3\leq d\leq 5$. Let $\{(u_n,v_n)\}_n$ be a bounded sequence in $\Hc^1$. Then there exist $(u,v)\in \Hc^1$ and a subsequence still denoted by $\{(u_n,v_n)\}_n$ such that $(u_n,v_n) \rightharpoonup (u,v)$ weakly in $\Hc^1$ and
	\begin{align} \label{P-un-vn}
	\Pb(u_n,v_n)\rightarrow \Pb(u,v) \text{ as } n\rightarrow \infty.
	\end{align}
\end{lemma}

\begin{proof}
	Let $\{(u_n,v_n)\}_n$ be a bounded sequence in $\Hc^1$. Then $\{u_n\}_n$ and $\{v_n\}_n$ are bounded sequences in $H^1(\Rb^d)$. Thus there exist $u,v \in H^1(\Rb^d)$ and subsequences still denoted by $\{u_n\}_n$ and $\{v_n\}_n$ such that 
	\begin{itemize}
		\item $u_n\rightharpoonup u$ and $v_n\rightharpoonup v$ weakly in $H^1(\Rb^d)$.
		\item $u_n \rightarrow u$ and $v_n \rightarrow v$ strongly in $L^q_{\loc}(\Rb^d)$ for all $q\geq 1$ and $q<\frac{2d}{d-2}$ if $d\geq 3$.
	\end{itemize}
	The idea is to use the smallness of $|x|^{-\alpha}$ for large $|x|$ and the strong convergence in $L^q_{\loc}(\Rb^d)$ for small $|x|$. More precisely, let $\vareps>0$. We first write
	\[
	\Pb(u_n,v_n) - \Pb(u,v) = \rea \int_{\Rb^d} |x|^{-\alpha} (u_n^2 - u^2)\overline{v}_n \dd x + \rea \int_{\Rb^d} |x|^{-\alpha} u^2 (\overline{v}_n-\overline{v})\dd x =: (\text{I}) + (\text{II}).
	\]
	For $R>0$ depending on $\vareps$ to be chosen shortly, we write
	\[
	(\text{I}) = \rea \int_{|x|\leq R} |x|^{-\alpha} (u_n^2-u^2)\overline{v}_n \dd x + \rea\int_{|x|\geq R} |x|^{-\alpha} (u_n^2-u^2)\overline{v}_n \dd x = (\text{I}_1) + (\text{I}_2).
	\]
	
	$\bullet$ For $(\text{I}_2)$, we have 
	\[
	|(\text{I}_2)| \leq R^{-\alpha} \|u_n+u\|_{L^3} \|u_n-u\|_{L^3} \|v_n\|_{L^3} \leq CR^{-\alpha} =\frac{\vareps}{4}
	\]
	provided that $R=\left(\frac{\vareps}{4C}\right)^{-\frac{1}{\alpha}}$, where we have used the Sobolev embedding $H^1(\Rb^d) \subset L^3(\Rb^d)$ for $1\leq d\leq 5$.
	
	$\bullet$ For $(\text{I}_1)$, we estimate
	\[
	|(\text{I}_1)| \leq \||x|^{-\alpha}\|_{L^\gamma(|x|\leq R)} \|u_n+u\|_{L^\rho} \|u_n-u\|_{L^\rho(|x|\leq R)} \|v_n\|_{L^\rho},
	\]
	where $1\leq \gamma, \rho\leq \infty$ are such that $1=\frac{1}{\gamma} + \frac{3}{\rho}$. We choose $\frac{1}{\gamma}=\frac{\alpha}{d}+\nu$ for some $\nu >0$ to be determined shortly, hence $	\rho=\frac{3d}{d-\alpha-d\nu}>2$. With this choice, we have 
	\[
	\||x|^{-\alpha}\|_{L^\gamma(|x|\leq R)} \leq C R^{d\nu}
	\]
	for some universal constant $C>0$.
	
	If $d=1,2$, we use the Sobolev embedding $H^1(\Rb^d)\subset L^q(\Rb^d)$ for $q\geq 2$ and the fact that $u_n \rightarrow u$ strongly in $L^q(|x|\leq R)$ for $q\geq 1$ to have for $n$ large enough depending on $\vareps$,
	\begin{align} \label{small-ball}
	|(\text{I}_1)| \leq C R^{d\nu}\|u_n-u\|_{L^\rho(|x|\leq R)} <\frac{\vareps}{4}.
	\end{align}
	
	If $3\leq d\leq 5$, we observe that $\rho<\frac{2d}{d-2}$ for $\nu>0$ sufficiently small due to $\alpha<\frac{6-d}{2}$. Hence, the Sobolev embedding $H^1(\Rb^d)\subset L^q(\Rb^d)$ for $2\leq q\leq \frac{2d}{d-2}$ and the convergence $u_n\rightarrow u$ strongly in $L^q(|x|\leq R)$ for all $1\leq q <\frac{2d}{d-2}$ imply \eqref{small-ball} for $n$ large enough depending on $\vareps$.
	 
	Collecting the above estimates, we have for $n$ sufficiently large depending on $\vareps$, $|(\text{I})|<\frac{\vareps}{2}$. A similar argument goes for $(\text{II})$ and we obtain for $n$ large enough depending on $\vareps$,
	\[
	|\Pb(u_n,v_n)-\Pb(u,v)|<\vareps.
	\]
	As $\vareps>0$ is arbitrary, we conclude the proof.
\end{proof}

\begin{proof}[Proof of Proposition \ref{prop-GN-ineq}]
	The proof follows from an idea of Weinstein \cite{Weinstein} (see also \cite{Farah}). However, instead of reducing the problem to radial setting using Schwarz' symmetrization, we work directly on general $\Hc^1$-functions. We only consider $(u,v) \ne (0,0)$ otherwise the result holds trivially. Define the Weinstein functional
	\[
	\Wb(u,v):= \frac{\Pb(u,v)}{\left(\Kb(u,v)\right)^{\frac{d+2\alpha}{4}} \left(\Mb(u,v)\right)^{\frac{6-d-2\alpha}{4}}},
	\]
	thus
	\[
	C_{\GN}:= \sup \left\{\Wb(u,v) : (u,v)\in \Hc^1\backslash \{(0,0)\}\right\}.
	\]
	Thanks to \eqref{gn-ineq}, we observe that $0<C_{\GN}<\infty$. Let $\{(u_n,v_n)\}_n$ be an optimizing sequence for $C_{\GN}$. As the Weinstein functional is invariant under the scaling $(\mu u(\lambda x), \mu v(\lambda x))$ with $\mu, \lambda>0$, we can use a suitable choice of scaling to assume (without loss of generality) that $\Mb(u_n,v_n) = \Kb(u_n,v_n) =1$ for all $n$ 
	and $\Pb(u_n,v_n) \rightarrow C_{\GN}$ as $n \to \infty$. Thus $\{(u_n,v_n)\}_n$ is a bounded sequence in $\Hc^1$. By Lemma \ref{lem-comp}, there exist $(u,v)$ and a subsequence still denoted by $\{(u_n,v_n)\}_n$ such that $(u_n,v_n)\rightharpoonup (u,v)$ weakly in $\Hc^1$ and $\Pb(u_n,v_n) \rightarrow \Pb(u,v)$ as $n\rightarrow \infty$. By the weak convergence in $\Hc^1$, we have
	\[
	\Kb(u,v) \leq \liminf_{n\rightarrow \infty} \Kb(u_n,v_n) =1, \quad \Mb(u,v) \leq \liminf_{n\rightarrow \infty} \Mb(u_n,v_n)=1.
	\]
	We infer that
	\[
	C_{\GN}=\lim_{n\rightarrow \infty} \Pb(u_n,v_n) = \Pb(u,v) \leq \Wb(u,v) \leq C_{\GN}
	\]
	hence
	\[
	\Kb(u,v)=\Mb(u,v)=1, \quad \Pb(u,v)=C_{\GN}
	\]
	or $(u,v)$ is an optimizer for $C_{\GN}$. It follows that
	\begin{align} \label{deri-W}
	\left.\frac{d}{d\vareps}\right|_{\vareps=0} \Wb(u+\vareps \chi, v+\vareps \vartheta) =0, \quad \forall \chi, \vartheta \in C^\infty_0(\Rb^d).
	\end{align}
	A direct computation yields
	\begin{align*}
	&\frac{1}{\Pb(u,v)} \rea \int_{\Rb^d} |x|^{-\alpha} \left( 2 \overline{u} v \overline{\chi} +  u^2 \overline{\vartheta}\right) \dd x \\
	&\mathrel{\phantom{===}} \phantom{} -\frac{d+2\alpha}{4 \Kb(u,v)} \rea \int_{\Rb^d} \left(-2\Delta u \overline{\chi} - 2\kappa \Delta v \overline{\vartheta}\right) \dd x - \frac{6-d-2\alpha}{4 \Mb(u,v)} \rea \int_{\Rb^d} \left(2u \overline{\chi} + 4 v \overline{\vartheta}\right) \dd x =0.
	\end{align*}	
	Testing again with $(\ic \chi, \ic \vartheta)$ instead of $(\chi,\vartheta)$, we obtain the same equality but with the imaginary part instead of the real one. In particular, we get
	\[
	\left\{
	\renewcommand*{\arraystretch}{1.5}
	\begin{array}{l}
	\frac{d+2\alpha}{2\Kb(u,v)} \Delta u -\frac{6-d-2\alpha}{2 \Mb(u,v)} u +\frac{2}{\Pb(u,v)} |x|^{-\alpha} \overline{u} v=0, \\
	\frac{d+2\alpha}{2\Kb(u,v)} \kappa \Delta v- \frac{6-d-2\alpha}{\Mb(u,v)} v + \frac{1}{\Pb(u,v)}|x|^{-\alpha} u^2 =0,
	\end{array}
	\right.
	\]
	in the weak sense. By setting 
	\[
	u(x) = \mu \varphi(\lambda x), \quad v(x)=\mu \psi(\lambda x)
	\]
	with
	\begin{align*}
	\lambda^2 = \frac{6-d-2\alpha}{d+2\alpha} \frac{\Kb(u,v)}{\Mb(u,v)}, \quad \mu = \frac{6-d-2\alpha}{2\lambda^\alpha} \frac{\Pb(u,v)}{\Mb(u,v)},
	\end{align*}
	we see that $\Wb(\varphi,\psi)=\Wb(u,v)=C_{\GN}$ or $(\varphi,\psi)$ is an optimizer for $C_{\GN}$, and $(\varphi,\psi)$ is a non-trivial solution to \eqref{equ-var-psi}.
	
	Now let $|\varphi|^*$ and $|\psi|^*$ be the Schwarz symmetric rearrangement of $|\varphi|$ and $|\psi|$ respectively. We have $\||\varphi|^*\|_{L^2}=\|\varphi\|_{L^2}$ and, by Polya-Szeg\"o's inequality, 
	\[
	\|\nabla |\varphi|^*\|_{L^2} \leq \|\nabla |\varphi|\|_{L^2} \leq \|\nabla \varphi\|_{L^2}. 
	\]
	By the extended Hardy-Littlewood inequality (see \cite[Theorem 1]{BH}), we also have
	\[
	\Pb(\varphi,\psi) \leq \int_{\Rb^d} |x|^{-\alpha} |\varphi|^2 |\psi| \dd x \leq \int_{\Rb^d} |x|^{-\alpha} (|\varphi|^*)^2 |\psi|^* \dd x = \Pb(|\varphi|^*, |\psi|^*).
	\]
	In particular, $\Wb(\varphi, \psi) \leq \Wb(|\varphi|^*, |\psi|^*)$ or $(|\varphi|^*,|\psi|^*)$ is also an optimizer for $C_{\GN}$. Thus we can assume, without loss of generality, that $\varphi$ and $\psi$ are non-negative, radially symmetric, and radially decreasing. It follows that
	\[
	\left\{
	\renewcommand*{\arraystretch}{1.3}
	\begin{array}{l}
	\frac{1}{2}\Delta \varphi -\varphi =-|x|^{-\alpha} \overline{\varphi} \psi \leq 0,\\
	\frac{\kappa}{2} \Delta \psi -2\psi =-\frac{1}{2} |x|^{-\alpha} \varphi^2 \leq 0,
	\end{array}
	\right. \text{ on } \Rb^d.
	\]
	By the maximum principle (see e.g., \cite[Theorem 3.5]{GT}), both $\varphi$ and $\psi$ are positive. The proof is complete.
\end{proof}

The existence of nonlinear ground states (see Definition \ref{def-non-gs}) is proved in Proposition \ref{prop-GN-ineq}. We collect some properties of solutions of \eqref{equ-var-psi} in the following lemma.

\begin{lemma} \label{lem-var-psi}
	Let $(\varphi,\psi)$ be a non-trivial solution to \eqref{equ-var-psi}. Then the following identities hold:
	\begin{align}
	\frac{1}{2} \Kb(\varphi,\psi) +\Mb(\varphi,\psi) - \frac{3}{2} \Pb(\varphi,\psi) &=0, \label{iden-1}\\
	\frac{d-2}{4} \Kb(\varphi,\psi) + \frac{d}{2} \Mb(\varphi,\psi) - \frac{d-\alpha}{2} \Pb(\varphi,\psi) &=0. \label{iden-2}
	\end{align}
	In particular, we have
	\begin{align} \label{iden-var-psi}
	\Kb(\varphi,\psi) = \frac{d+2\alpha}{2} \Pb(\varphi,\psi) = \frac{2(d+2\alpha)}{6-d-2\alpha} \Mb(\varphi,\psi),
	\end{align}
\end{lemma}

\begin{proof}
	Multiplying the first equation with $\overline{\varphi}$ and the second one with $\overline{\psi}$, integrating over $\Rb^d$, and taking the real part, we get \eqref{iden-1}. Multiplying the first equation with $x\cdot \nabla \overline{\varphi}$ and the second one with $x \cdot \nabla \overline{\psi}$, integrating over $\Rb^d$, and taking the real part, we obtain \eqref{iden-2}. From \eqref{iden-1} and \eqref{iden-2}, we infer \eqref{iden-var-psi}. 
\end{proof}

\section{Global existence} \label{S-gwp}
\setcounter{equation}{0}

\subsection{Local wellposedness}
In this paragraph, we explain how to get the local existence of $\Hc^1$ solutions to \eqref{INLS} using an abstract theory of Cazenave \cite[Theorem 3.3.9, Remark 3.3.12, and Theorem 4.3.1]{Cazenave}. Let
\[
\vec{u}:=(u_1, u_2), \quad \Ac \vec{u} := (\alpha_1 \Delta u_1, \alpha_2 \Delta u_2)
\]
and consider
\begin{align}\label{NLS}
	\ic \partial_t \vec{u} + \Ac \vec{u} + g(\vec{u})=0.
\end{align}

\begin{proposition}[\cite{Cazenave}] \label{prop-lwp-nls}
	Let $g=g_1+g_2+g_3$ satisfy the following conditions: 
	\begin{itemize}
		\item[(1)] For $j=1,2,3$, $g_j \in C(\Hc^1, \Hc^{-1})$, $g_j(\vec{0}) = \vec{0}$, and there exists $G_j \in C^1(\Hc^1,\Rb)$ such that $g_{jk}=\partial_kG_j$ for $k=1,2$, where $g_j=(g_{j1},g_{j2})$ and $\partial_k$ stands for the Fr\'echet derivative with respect to the $k$-th variable. 
		\item[(2)] For $j=1,2,3$, there exist $r_j, \rho_j \in \left[2,\frac{2d}{d-2}\right)$ if $d\geq 2$ or $r_j,\rho_j \in [2,\infty]$ if $d=1$ such that for every $M>0$, there exists $C(M)>0$ so that
		\[
		\|g_j (\vec{u}) - g_j (\vec{v})\|_{\Lc^{\rho'_j}} \leq C(M) \|\vec{u} -\vec{v}\|_{\Lc^{r_j}}
		\]
		for all $\vec{u}, \vec{v} \in \Hc^1$ satisfying $\|\vec{u}\|_{\Hc^1} +\|\vec{v}\|_{\Hc^1} \leq M$. Here $(\rho_j, \rho'_j)$ is a H\"older conjugate pair.	
		\item[(3)] There exists $(\beta_1, \beta_2) \in \Rb^2$ such that for $j=1,2,3$ and every $\vec{u} \in \Hc^1$, 
		\[
		\ima(\beta_1 g_{j1}(\vec{u}) \overline{u}_1 + \beta_2 g_{j2}(\vec{u}) \overline{u}_2) =0 \text{ a.e. in } \Rb^d.
		\]
	\end{itemize}
	Then for any $\vec{u}_0\in \Hc^1$, there exists a unique maximal solution 
	\[
	\vec{u} \in C((-T_*,T^*), \Hc^1) \cap C^1((-T_*,T^*),\Hc^{-1})
	\]
	to \eqref{NLS} with initial data $\left.u\right|_{t=0} = \vec{u}_0$. The maximal time satisfies the blow-up alternative: if $T^*<\infty$ (resp. $T_*<\infty$), then $\lim_{t\nearrow T^*}\|\vec{u}(t)\|_{\Hc^1}=\infty$ (resp. $\lim_{t\searrow -T_*}\|\vec{u}(t)\|_{\Hc^1}=\infty$). In addition, there are conservation laws of mass and energy, i.e.,
	\begin{align*}
		\Mb(\vec{u}(t)) &= \beta_1 \|u_1(t)\|^2_{L^2} + \beta_2 \|u_2(t)\|^2_{L^2} = \Mb(\vec{u}_0), \tag{Mass} \\
		\Eb(\vec{u}(t)) &= \alpha_1 \|\nabla u_1(t)\|^2_{L^2} + \alpha_2\|\nabla u_2(t)\|^2_{L^2} - 2\sum_{j=1}^3 G_j(\vec{u}(t)) = \Eb(\vec{u}_0), \tag{Energy} 
	\end{align*}
	for all $t\in (-T_*,T^*)$.
\end{proposition}

\begin{proof}[Proof of Proposition \ref{prop-lwp-inls}]
	We apply Proposition \ref{prop-lwp-nls} with 
	\[
	(\alpha_1,\alpha_2)=\left(\frac{1}{2},\frac{\kappa}{2}\right), \quad (\beta_1, \beta_2) = (1,2)
	\]
	and
	\begin{align*}
		g_1(u,v) &= \left(0,-\gamma v\right), \\ 
		g_2(u,v) &= \left( \mathds{1}_{B_1} |x|^{-\alpha} \overline{u} v, \frac{1}{2} \mathds{1}_{B_1} |x|^{-\alpha} u^2\right), \\ g_3(u,v) &= \left(\mathds{1}_{B_1^c} |x|^{-\alpha} \overline{u} v, \frac{1}{2} \mathds{1}_{B_1^c} |x|^{-\alpha} u^2  \right),
	\end{align*}
	where 
	\[
	B_1 := \left\{x \in \Rb^d : |x| <1\right\}, \quad B_1^c:= \Rb^d \backslash B_1.
	\]
	We have
	\begin{align*}
		G_1(u,v) &= -\frac{\gamma}{2} \|v\|^2_{L^2}, \\
		G_2(u,v) &= \frac{1}{2} \rea \int_{B_1} |x|^{-\alpha} u^2\overline{v} \dd x, \\
		G_3(u,v) &= \frac{1}{2} \rea \int_{B_1^c} |x|^{-\alpha} u^2\overline{v} \dd x.
	\end{align*}
	One readily check that the conditions (1) and (3) are fulfilled. It remains to check the condition (2). 
	
	$\bullet$ For $g_1$, we simply take $r_1=\rho_1=2$ and get
	\[
	\|g_1(u_1,v_1)-g_1(u_2,v_2)\|_{\Lc^2} = \gamma \|v_1-v_2\|_{L^2} \leq \gamma \|(u_1,v_1)-(u_2,v_2)\|_{\Lc^2}.
	\]
	
	$\bullet$ For $g_2$, we write
	\begin{align*}
		\|g_2(u_1,v_1)-g_2(u_2,v_2)\|_{\Lc^{\rho'_2}} =\|\mathds{1}_{B_1} |x|^{-\alpha}(\overline{u}_1 v_1 - \overline{u}_2 v_2)\|_{L^{\rho_2'}} + \|\mathds{1}_{B_1}|x|^{-\alpha}(u_1^2-u_2^2)\|_{L^{\rho_2'}} =: (\text{I}) + (\text{II}).
	\end{align*}
	We only estimate $(\text{I})$ since the one for $(\text{II})$ is treated in a similar manner. We have
	\begin{align*}
		(\text{I}) \leq \|\mathds{1}_{B_1} |x|^{-\alpha}(\overline{u}_1-\overline{u}_2) v_1\|_{L^{\rho_2'}} + \|\mathds{1}_{B_1} |x|^{-\alpha}\overline{u}_2(v_1-v_2)\|_{L^{\rho_2'}} =:(\text{I}_1) + (\text{I}_2).
	\end{align*}
	By H\"older's inequality, 
	\begin{align*}
		(\text{I}_1) \leq \|\mathds{1}_{B_1} |x|^{-\alpha}\|_{L^{\gamma_2}} \|v_1\|_{L^{\rho_2}} \|u_1-u_2\|_{L^{\rho_2}}
	\end{align*}
	with $\gamma_2, \rho_2 \in [1,\infty]$ satisfying
	\[
	\frac{1}{\rho_2'} = \frac{1}{\gamma_2} + \frac{2}{\rho_2} \text{ or } 1-\frac{1}{\gamma_2} = \frac{3}{\rho_2}.
	\]
	We first choose $\frac{1}{\gamma_2} = \frac{\alpha}{d}+\vareps$ with $\vareps>0$ to be chosen later. This ensures $\|\mathds{1}_{B_1} |x|^{-\alpha}\|_{L^{\gamma_2}}<\infty$. It follows that
	\[
	\rho_2=\frac{3d}{d-\alpha-d\vareps} >2
	\]
	for any $\vareps>0$. 
	
	For $d=1,2$, we use the Sobolev embedding $H^1(\Rb^d) \subset L^q(\Rb^d)$ for all $2\leq q<\infty$ to get
	\[
	(\text{I}_1) \leq C \|v_1\|_{H^1} \|u_1-u_2\|_{L^{\rho_2}} \leq C(M)\|u_1-u_2\|_{L^{\rho_2}} \leq C(M)\|(u_1,v_1)-(u_2,v_2)\|_{\Lc^{\rho_2}}
	\]
	provided that $\|(u_1,v_1)\|_{\Hc^1}+\|(u_2,v_2)\|_{\Hc^1}\leq M$. 
	
	For $3\leq d\leq 5$, we observe that $\rho_2 <\frac{2d}{d-2}$ for $\vareps>0$ small enough due to $\alpha<\frac{6-d}{2}$. The Sobolev embedding $H^1(\Rb^d)\subset L^q(\Rb^d)$ for all $2\leq q \leq \frac{2d}{d-2}$ yields
	\begin{align*}
		(\text{I}_1) \leq C(M)\|(u_1,v_1)-(u_2,v_2)\|_{\Lc^{\rho_2}}
	\end{align*}
	provided that $\|(u_1,v_1)\|_{\Hc^1}+\|(u_2,v_2)\|_{\Hc^1}\leq M$. The term $(\text{I}_2)$ is treated similarly, and thus the condition (2) is satisfied by $g_2$ with $\rho_2=r_2$.
	
	$\bullet$ For $g_3$, we choose $\rho_3=r_3=3$ and we estimate
	\begin{align*}
		\|g_3(u_1,v_1)-g_3(u_2,v_2)\|_{\Lc^{\frac{3}{2}}} =\|\mathds{1}_{B^c_1} |x|^{-\alpha}(\overline{u}_1 v_1 - \overline{u}_2 v_2)\|_{L^{\frac{3}{2}}} + \|\mathds{1}_{B^c_1}|x|^{-\alpha}(u_1^2-u_2^2)\|_{L^{\frac{3}{2}}} =: (\text{III}) + (\text{IV}).
	\end{align*}
	As above, we only consider $(\text{III})$. We write
	\begin{align*}
		(\text{III}) \leq \|\mathds{1}_{B_1^c} |x|^{-\alpha}(\overline{u}_1-\overline{u}_2) v_1\|_{L^{\frac{3}{2}}} + \|\mathds{1}_{B_1^c} |x|^{-\alpha}\overline{u}_2(v_1-v_2)\|_{L^{\frac{3}{2}}} =:(\text{III}_1) + (\text{III}_2).
	\end{align*}
	By H\"older's inequality,  we have
	\begin{align*}
		(\text{III}_1) \leq \|\mathds{1}_{B^c_1} |x|^{-\alpha}\|_{L^\infty} \|v_1\|_{L^3} \|u_1-u_2\|_{L^3}.
	\end{align*}
	As $H^1(\Rb^d)\subset L^3(\Rb^d)$ for $1\leq d\leq 5$, we get
	\[
	(\text{III}_1)\leq C\|v_1\|_{H^1} \|u_1-u_2\|_{L^3} \leq C(M)\|(u_1,v_1)-(u_2,v_2)\|_{\Lc^3}
	\]
	provided that $\|(u_1,v_1)\|_{\Hc^1} + \|(u_2,v_2)\|_{\Hc^1}\leq M$. Thus the condition (2) is also satisfied by $g_3$.
\end{proof}

\subsection{Global existence}
We now prove the global existence results for \eqref{INLS} given in Proposition \ref{prop-gwp}. The proof follows from a standard argument (see e.g. \cite{Farah}), but for the sake of completeness, we provide some details. It suffices to show
\begin{align} \label{blow-alte}
	\sup_{t\in (-T_*,T^*)} \Kb(u(t),v(t)) \leq C
\end{align}
which together with the blow-up alternative (see Proposition \ref{prop-lwp-inls}) yields $T_*=T^*=\infty$.

\begin{proof}[\underline{Proof of Proposition \ref{prop-gwp} in the mass-subcritical case}]
	By \eqref{GN-ineq} and Young's inequality with $d+2\alpha<4$, we have
	\begin{align*}
	\Pb(u(t),v(t)) &\leq C_{\GN} \left( \Kb(u(t),v(t))\right)^{\frac{d+2\alpha}{4}} \left( \Mb(u(t),v(t))\right)^{\frac{6-d-2\alpha}{4}} \\
	&\leq \frac{1}{4} \Kb(u(t),v(t)) + C(C_{\GN}, \Mb(u(t),v(t))), \quad \forall t\in (-T_*,T^*).
	\end{align*}
	Using
	\begin{align} \label{est-gamma}
	-\gamma \|v(t)\|^2_{L^2} \leq \frac{|\gamma|}{2} \Mb(u(t),v(t)), 
	\end{align}
	we get
	\begin{align*}
	\Eb(u(t),v(t)) + \frac{|\gamma|}{2}\Mb(u(t),v(t)) &\geq	\Eb(u(t),v(t))-\gamma \|v(t)\|^2_{L^2} \\
	&=\frac{1}{2}\Kb(u(t),v(t)) - \Pb(u(t),v(t)) \\
	&\geq \frac{1}{4} \Kb(u(t),v(t)) - C(C_{\GN}, \Mb(u(t),v(t)))
	\end{align*}
	which together with the mass and energy conservation yields \eqref{blow-alte}.
\end{proof}

\begin{proof}[\underline{Proof of Proposition \ref{prop-gwp} in the mass-critical case}]
	We first observe that in the mass-critical case $\alpha=\frac{4-d}{2}$, by using \eqref{iden-var-psi}, we have
	\[
	C_{\GN} =\frac{1}{2 \left(\Mb(\varphi,\psi)\right)^{1/2}}.
	\]
	By \eqref{GN-ineq}, we get
	\[
	\Pb(u(t),v(t))\leq \frac{1}{2} \left(\frac{\Mb(u(t),v(t))}{\Mb(\varphi,\psi)}\right)^{1/2} \Kb(u(t),v(t))
	\]
	which along with \eqref{est-gamma} yield
	\[
	\frac{1}{2}\left(1-\left(\frac{\Mb(u(t),v(t))}{\Mb(\varphi,\psi)}\right)^{1/2}\right) \Kb(u(t),v(t)) \leq \Eb(u(t),v(t)) + \frac{|\gamma|}{2} \Mb(u(t),v(t)), \quad \forall t\in (-T_*,T^*).
	\]
	By the conservation of mass and energy and using \eqref{cond-gwp-crit}, we get \eqref{blow-alte}.
\end{proof}

\begin{proof}[\underline{Proof of Proposition \ref{prop-gwp} in the mass-supercritical case}]
	By the Gagliardo--Nirenberg inequality \eqref{GN-ineq}, we have 
	\begin{align*}
	(\Eb(u(t),v(t)) &- \gamma \|v(t)\|^2_{L^2}) \left(\Mb(u(t),v(t))\right)^\sigma \\
	&= \frac{1}{2} \Kb(u(t),v(t)) \left(\Mb(u(t),v(t))\right)^\sigma - \Pb(u(t),v(t)) \left(\Mb(u(t),v(t))\right)^\sigma \\
	&\geq \frac{1}{2} \Kb(u(t),v(t)) \left(\Mb(u(t),v(t))\right)^\sigma - C_{\GN} \left(\Kb(u(t),v(t))\right)^{\frac{d+2\alpha}{4}} \left(\Mb(u(t),v(t))\right)^{\frac{6-d-2\alpha}{4}+\sigma} \\
	&= g \left(\Kb(u(t),v(t)) \left(\Mb(u(t),v(t))\right)^\sigma\right), \quad \forall t\in (-T_*,T^*),
	\end{align*}
	where
	\begin{align} \label{defi-g}
	g(\lambda):= \frac{1}{2} \lambda - C_{\GN} \lambda^{\frac{n+2\alpha}{4}}.
	\end{align}
	On the other hand, using \eqref{est-gamma}, we have
	\[
	\Eb(u(t),v(t)) - \gamma \|v(t)\|^2_{L^2} \leq \left\{
	\begin{array}{ccl}
	\Eb(u(t),v(t)) &\text{if}& \gamma \geq 0, \\
	\Eb(u(t),v(t))+\frac{|\gamma|}{2} \Mb(u(t),v(t)) &\text{if}& \gamma<0,
	\end{array}
	\right.
	\]
	or
	\begin{align} \label{est-Hb}
	\Eb(u(t),v(t)) - \gamma \|v(t)\|^2_{L^2} \leq \Hb(u(t),v(t)), \quad \forall t \in (-T_*,T^*). 
	\end{align}
	By the conservation of mass and energy, we get
	\begin{align*} 
	g\left(\Kb(u(t),v(t)) \left(\Mb(u(t),v(t))\right)^\sigma\right) \leq \Hb(u_0,v_0) \left(\Mb(u_0,v_0)\right)^\sigma,
	\end{align*}
	which together with \eqref{cond-gwp-supe-1} imply
	\begin{align} \label{est-g}
	g\left(\Kb(u(t),v(t)) \left(\Mb(u(t),v(t))\right)^\sigma\right) < \Eb_0(\varphi,\psi) \left(\Mb(\varphi,\psi)\right)^\sigma, \quad \forall t\in (-T_*,T^*).
	\end{align}
	Thanks to the continuity of the map $t\mapsto \Kb(u(t),v(t)) \left(\Mb(u(t),v(t))\right)^\sigma$ (as $(u,v)\in C((-T_*,T^*), \Hc^1)$) and the fact that
	\begin{align} \label{g-GN}
		g\left( \Kb(\varphi,\psi) \left(\Mb(\varphi,\psi)\right)^\sigma\right) = \frac{d+2\alpha-4}{2(d+2\alpha)} \Kb(\varphi,\psi) \left(\Mb(\varphi,\psi)\right)^\sigma = \Eb_0(\varphi,\psi) \left(\Mb(\varphi,\psi)\right)^\sigma, 
	\end{align}
	we deduce from \eqref{cond-gwp-supe-2} that
	\begin{align} \label{claim-gwp-supe}
	\Kb(u(t),v(t)) \left( \Mb(u(t),v(t))\right)^\sigma < \Kb(\varphi,\psi) \left(\Mb(\varphi,\psi)\right)^\sigma, \quad \forall t\in (-T_*,T^*). 
	\end{align}	
	From \eqref{claim-gwp-supe} and the mass conservation, we obtain \eqref{blow-alte}.
\end{proof}

\section{Energy scattering}
\label{S-scat} 
\setcounter{equation}{0}
The purpose of this section is to prove the following asymptotic behavior (or energy scattering) of $\Hc^1$-solutions to \eqref{INLS} in the mass-supercritical regime.

\subsection{Dispersive and Strichartz estimates}
Let $\beta_1, \beta_2 \in \Rb$ and denote $\Sc(t):= \ee^{\ic t(\beta_1\Delta -\beta_2)}$ the Schr\"odinger operator. We have the following dispersive estimates (see e.g., \cite{Cazenave}): for $2\leq r\leq \infty$,
\begin{align} \label{dis-est}
	\|\Sc(t) f\|_{L^r_x} \lesssim |t|^{-\left(\frac{d}{2}-\frac{d}{r}\right)} \|f\|_{L^{r'}_x}, \quad \forall t\ne 0
\end{align}
for all $f\in L^{r'}_x(\Rb^d)$. 

Let $0\leq s< \min\left\{1,\frac{d}{2}\right\}$. A pair $(q,r)$ is called $\dot{H}^s$-admissible if 
\[
\frac{2}{q}+\frac{d}{r} =\frac{d}{2}-s
\]
and
\begin{align} \label{cond-r}
	\left\{
	\begin{array}{ccl}
		\frac{2d}{d-2s}<r<\frac{2d}{d-2} &\text{if}& d\geq 3, \\
		\frac{2}{1-s}<r<\infty &\text{if}& d=2, \\
		\frac{2}{1-2s}<r<\infty &\text{if} & d=1.
	\end{array}
	\right.
\end{align}
The set of all $\dot{H}^s$-admissible pairs is denoted by $\Ac_s$. Similarly, a pair $(q,r)$ is called $\dot{H}^{-s}$-admissible if 
\[
\frac{2}{q}+\frac{d}{r}=\frac{d}{2}+s
\]
and $r$ satisfies \eqref{cond-r}. We denote by $\Ac_{-s}$ the set of all $\dot{H}^{-s}$-admissible pairs. 
\begin{proposition}[Strichartz estimates \cite{Cazenave, KT, Fos, Gue}]
	Then for $d\geq 1$ and $0\leq s <\min\left\{1, \frac{d}{2}\right\}$, we have
	\[
	\|\Sc(t) f\|_{L^q_t(\Rb, L^r_x)} \lesssim \|f\|_{\dot{H}^s_x}
	\]
	for any $f\in \dot{H}^s(\Rb^d)$ and any $(q,r) \in \Ac_s$. Moreover, for any interval $I\subset \Rb$ containing $0$, there exists $C>0$ independent of $I$ such that 
	\[
	\left\|\int_0^t \Sc(t-\tau) F(\tau) \dd\tau\right\|_{L^q_t(I, L^r_x)} \lesssim \|F\|_{L^{a'}_t(I,L^{b'}_x)}
	\]
	for any $F\in L^{a'}_t(I, L^{b'}_x(\Rb^d))$, any $(q,r) \in \Ac_s$, and any $(a,b) \in \Ac_{-s}$. 
\end{proposition}

\subsection{Nonlinear estimates}
Let $\theta>0$ be a small parameter. We introduce the following exponents
\begin{align*}
	q&=\frac{4}{2-\theta}, &r&= \frac{2d}{d-2+\theta}, \\
	\overline{q}&=\frac{4(2-\theta)}{d-2+2\alpha-\theta(d-5+2\alpha)}, &\overline{r}&=\frac{2d(2-\theta)}{d+2-2\alpha -\theta(5-2\alpha)}, \\
	\overline{a}&=\frac{4(2-\theta)}{6-d-2\alpha+\theta}, &\tilde{a}&=\frac{2(3-\theta)}{2d-6+4\alpha-\theta(d-4+2\alpha)}, \\
	\hat{a}&=\frac{2(3-\theta)}{6-d-2\alpha}, &\hat{r}&=\frac{d(3-\theta)}{d-\alpha-\theta(2-\alpha)}, \\
	\overline{m}_\pm&=\frac{d}{2-\alpha \mp d\theta}. & & 
\end{align*}
\begin{lemma} \label{lem-expo}
	Let $2\leq d \leq 5$, $0<\alpha<2$, and $\frac{4-d}{2}<\alpha <\frac{6-d}{2}$. Then there exists $\theta>0$ sufficiently small so that
	\begin{align*}
		\left\{
		\renewcommand*{\arraystretch}{1.3}
		\begin{array}{l}
			(q,r), (\overline{q},\overline{r}) \in \Ac_0, \\
			(\overline{a}, \overline{r}), (\hat{a},\hat{r}) \in \Ac_{s_c}, \\
			(\tilde{a}, \hat{r}) \in \Ac_{-s_c}, \\
			2<\overline{m}_\pm <\frac{2d}{d-2},
		\end{array}
		\right.
	\end{align*}
	where $s_c$ is the critical Sobolev exponent given in \eqref{sc}. 
\end{lemma}

\begin{proof}
	The proof follows from straightforward computations. Note that the condition \eqref{cond-r} follows by taking $\theta>0$ sufficiently small and using $\frac{4-d}{2}<\alpha<\frac{6-d}{2}$. 
\end{proof}

\begin{lemma} \label{lem-non-est-1}
	Let $3\leq d\leq 5$, $0<\alpha<\min\left\{2,\frac{d}{2}\right\}$, $\frac{4-d}{2}<\alpha<\frac{6-d}{2}$, and $I\subset \Rb$ be an interval. Let $\theta>0$ be a small parameter as in Lemma \ref{lem-expo}. Then we have
	\begin{align}
		\||x|^{-\alpha} \overline{u} v\|_{L^{\tilde{a}'}_t(I,L^{\hat{r}'}_x)} &\lesssim \|u\|^\theta_{L^\infty_t(I,H^1_x)} \|u\|^{1-\theta}_{L^{\hat{a}}_t(I, L^{\hat{r}}_x)} \|v\|_{L^{\hat{a}}_t (I,L^{\hat{r}}_x)}, \label{non-est-1} \\
		\||x|^{-\alpha} \overline{u} v\|_{L^{q'}_t(I,L^{r'}_x)} &\lesssim \|u\|^\theta_{L^\infty_t(I,H^1_x)} \|u\|^{1-\theta}_{L^{\overline{a}}_t(I, L^{\overline{r}}_x)} \|v\|_{L^{\overline{q}}_t (I,L^{\overline{r}}_x)}, \label{non-est-2} \\
		\|\nabla(|x|^{-\alpha} \overline{u} v)\|_{L^{q'}_t(I,L^{r'}_x)} &\lesssim \|u\|^\theta_{L^\infty_t(I,H^1_x)} \|u\|^{1-\theta}_{L^{\overline{a}}_t(I,L^{\overline{r}}_x)} \|\nabla v\|_{L^{\overline{q}}_t(I,L^{\overline{r}}_x)} + \|v\|^\theta_{L^\infty_t(I,H^1_x)} \|v\|^{1-\theta}_{L^{\overline{a}}_t(I,L^{\overline{r}}_x)} \|\nabla u\|_{L^{\overline{q}}_t(I,L^{\overline{r}}_x)}. \label{non-est-3}
	\end{align}
\end{lemma}

\begin{proof}
	We follow an argument of \cite{Campos, DK-NA}. By H\"older's  inequality, we have
	\[
	\||x|^{-\alpha}\overline{u} v\|_{L^{\hat{r}'}_x} \leq \||x|^{-\alpha}\|_{L^\gamma_x(A)} \|\overline{u} v\|_{L^\rho_x}
	\]
	provided that $\gamma, \rho>1$ is such that $\frac{1}{\hat{r}'}=\frac{1}{\gamma}+\frac{1}{\rho}$, where $A$ is either $B_1=B(0,1)$ the unit ball or $B_1^c=\Rb^d\backslash B_1$. To ensure $\||x|^{-\alpha}\|_{L^\gamma_x(A)}<\infty$, we take
	\[
	\frac{1}{\gamma} =\frac{\alpha}{d} \pm \theta^2
	\]
	with the plus sign for $A=B_1$ and the minus sign for $A=B_1^c$. Using the fact that
	\[
	\frac{1}{\rho}=\frac{1}{\hat{r}'}-\frac{1}{\gamma} = \frac{2d-2\alpha-(d-2)\theta}{d(3-\theta)} \mp \theta^2 =\frac{\theta}{\overline{m}_\pm} + \frac{2-\theta}{\hat{r}},
	\]
	we have
	\[
	\||x|^{-\alpha}\overline{u} v\|_{L^{\hat{r}'}_x} \lesssim \|u\|^\theta_{L^{\overline{m}_\pm}_x} \|u\|^{1-\theta}_{L^{\hat{r}}_x} \|v\|_{L^{\hat{r}}_x}.
	\]
	As $\frac{1}{\tilde{a}'} = \frac{2-\theta}{\hat{a}}$, H\"older's  inequality and Sobolev embedding with $2<\overline{m}_\pm<\frac{2d}{d-2}$  yield
	\begin{align*}
		\||x|^{-\alpha}\overline{u} v\|_{L^{\tilde{a}'}_t(I,L^{\hat{r}'}_x)} &\lesssim \|u\|^\theta_{L^\infty_t(I,L^{\overline{m}_\pm}_x)} \|u\|^{1-\theta}_{L^{\hat{a}}_t(I,L^{\hat{r}}_x)} \|v\|_{L^{\hat{a}}_t(I,L^{\hat{r}}_x)} \\
		&\lesssim \|u\|^\theta_{L^\infty_t(I,H^1_x)} \|u\|^{1-\theta}_{L^{\hat{a}}_t(I,L^{\hat{r}}_x)} \|v\|_{L^{\hat{a}}_t(I,L^{\hat{r}}_x)}.
	\end{align*}
	This shows \eqref{non-est-1}. We also have the following estimate which will be useful later
	\begin{align} \label{est-A}
		\||x|^{-\alpha} \overline{u} v\|_{L^{\hat{r}'}_x} \leq \||x|^{-\alpha}\|_{L^\gamma_x(A)} \|u\|^\theta_{L^{\overline{m}_\pm}_x} \|u\|^{1-\theta}_{L^{\hat{r}}_x} \|v\|_{L^{\hat{r}}_x}.
	\end{align}
	
	We next prove \eqref{non-est-3}. We have
	\[
	|\nabla(|x|^{-\alpha}\overline{u}v)| \leq |x|^{-\alpha} |\nabla(\overline{u}v)| + \alpha|x|^{-\alpha} ||x|^{-1}(\overline{u}v)|,
	\]
	hence
	\[
	\|\nabla(|x|^{-\alpha}\overline{u}v)\|_{L^{r'}_x} \leq \||x|^{-\alpha} \nabla(\overline{u}v)\|_{L^{r'}_x} + \alpha\||x|^{-\alpha}|x|^{-1}(\overline{u} v)\|_{L^{r'}_x}.
	\]
	We estimate
	\[
	\||x|^{-\alpha} \nabla (\overline{u}v)\|_{L^{r'}_x} \leq \||x|^{-\alpha}\|_{L^\gamma_x(A)} \|\nabla (\overline{u} v)\|_{L^\rho_x} 
	\]
	provided that $\gamma,\rho>1$ and $\frac{1}{r'}=\frac{1}{\gamma}+\frac{1}{\rho}$. To make $\||x|^{-\alpha}\|_{L^\gamma_x(A)} <\infty$, we take $\frac{1}{\gamma}=\frac{\alpha}{d}\pm \theta^2$ as before. In particular, we have
	\[
	\frac{1}{\rho}=\frac{1}{r'}-\frac{1}{\gamma}=\frac{d+2-2\alpha-\theta}{2d} \mp \theta^2.
	\]
	Similarly, we have
	\[
	\||x|^{-\alpha} |x|^{-1}(\overline{u}v)\|_{L^{r'}_x} \leq \||x|^{-\alpha}\|_{L^\gamma_x(A)} \||x|^{-1}(\overline{u} v)\|_{L^\rho_x} \lesssim \||x|^{-1}(\overline{u} v)\|_{L^\rho_x}.
	\]
	As $\alpha<\frac{d}{2}$, we see that $1<\rho<d$ by taking $\theta>0$ sufficiently small. Applying the following Hardy inequality (see e.g., \cite{OK}): for $1<\rho<d$,
	\[
	\||x|^{-1} f\|_{L^\rho_x} \leq \frac{\rho}{d-\rho} \|\nabla f\|_{L^\rho_x},
	\]
	we get
	\[
	\||x|^{-\alpha} |x|^{-1}(\overline{u}v)\|_{L^{r'}_x} \lesssim \|\nabla(\overline{u} v)\|_{L^\rho_x}.
	\]
	In particular, we obtain
	\begin{align*}
		\|\nabla(|x|^{-\alpha}\overline{u}v)\|_{L^{r'}_x} &\lesssim \|\nabla (\overline{u} v)\|_{L^\rho_x} \\
		&\lesssim \|\overline{u} \nabla v\|_{L^\rho_x} + \|v\nabla \overline{u}\|_{L^\rho_x} \\
		&\lesssim \|u\|^\theta_{L^{\overline{m}_\pm}_x} \|u\|^{1-\theta}_{L^{\overline{r}}_x} \|\nabla v\|_{L^{\overline{r}}_x} + \|v\|^\theta_{L^{\overline{m}_\pm}_x} \|v\|^{1-\theta}_{L^{\overline{r}}_x} \|\nabla u\|_{L^{\overline{r}}_x}, 
	\end{align*}
	where we have used the fact that $\frac{1}{\rho}=\frac{\theta}{\overline{m}_\pm}+\frac{2-\theta}{\overline{r}}$. Another application of the H\"older  inequality with $\frac{1}{q'}=\frac{1-\theta}{\overline{a}} +\frac{1}{\overline{q}}$ and Sobolev embedding  yields
	\begin{align*}
		\|\nabla(|x|^{-\alpha}\overline{u}v)\|_{L^{q'}_t(I,L^{r'}_x)} &\lesssim \|u\|^\theta_{L^\infty_t(I,L^{\overline{m}_\pm}_x)} \|u\|^{1-\theta}_{L^{\overline{a}}_t(I,L^{\overline{r}}_x)} \|\nabla v\|_{L^{\overline{q}}_t(I,L^{\overline{r}}_x)} + \|v\|^\theta_{L^\infty_t(I,L^{\overline{m}_\pm}_x)} \|v\|^{1-\theta}_{L^{\overline{a}}_t(I,L^{\overline{r}}_x)} \|\nabla u\|_{L^{\overline{q}}_t(I,L^{\overline{r}}_x)} \\
		&\lesssim  \|u\|^\theta_{L^\infty_t(I,H^1_x)} \|u\|^{1-\theta}_{L^{\overline{a}}_t(I,L^{\overline{r}}_x)} \|\nabla v\|_{L^{\overline{q}}_t(I,L^{\overline{r}}_x)} + \|v\|^\theta_{L^\infty_t(I,H^1_x)} \|v\|^{1-\theta}_{L^{\overline{a}}_t(I,L^{\overline{r}}_x)} \|\nabla u\|_{L^{\overline{q}}_t(I,L^{\overline{r}}_x)}
	\end{align*}
	which is \eqref{non-est-3}. The estimate \eqref{non-est-2} is treated similarly (even simpler) as for \eqref{non-est-3}. The proof is complete.
\end{proof}

\begin{lemma} \label{lem-non-est-2}
	Let $3\leq d\leq 5$, $0<\alpha<2$, $\frac{4-d}{2}<\alpha<\frac{6-d}{2}$, and $\frac{2d}{d+4} <m <\frac{2d}{d+2}$, we have
	\begin{align} \label{non-est-4}
		\||x|^{-\alpha} \overline{u} v\|_{L^m_x} \lesssim \|u\|_{H^1_x}\|v\|_{H^1_x}.
	\end{align}
\end{lemma}

\begin{proof}
	We estimate
	\begin{align} \label{est-theta}
		\||x|^{-\alpha} \overline{u} v\|_{L^m_x} \leq \||x|^{-\alpha}\|_{L^\gamma_x(A)} \|\overline{u} v\|_{L^\rho_x} \lesssim \|u\|_{L^{2\rho}_x} \|v\|_{L^{2\rho}_x} \lesssim \|u\|_{H^1_x} \|v\|_{H^1_x}
	\end{align}
	provided that $\gamma,\rho>1$, $\frac{1}{m}=\frac{1}{\gamma}+\frac{1}{\rho}$, $\||x|^{-\alpha}\|_{L^\gamma_x(A)} <\infty$, and $2\rho \in \left[2,\frac{2d}{d-2} \right]$, where $A=B_1$ or $B_1^c$. To make $\||x|^{-\alpha}\|_{L^\gamma_x(A)} <\infty$, we take $\gamma>1$ so that
	\[
	\frac{1}{\gamma}=\frac{\alpha}{d}\pm \theta^2
	\]
	with the plus sign for $A=B_1$ and the minus sign for $A=B_1^c$. It follows that
	\[
	\rho = \frac{dm}{d-\alpha m \mp \theta^2 d m}.
	\]
	For $\theta>0$ sufficiently small, the condition $2\rho \in \left[2,\frac{2d}{d-2}\right]$ is fulfilled provided that $\frac{d}{d+\alpha}<m<\frac{d}{d-2+\alpha}$. As $\frac{4-d}{2}<\alpha<\frac{6-d}{2}$, we infer that
	\[
	\frac{d}{d+\alpha} <\frac{2d}{d+4}, \quad \frac{2d}{d+2}<\frac{d}{d-2+\alpha}.
	\]
	Thus for $\frac{2d}{d+4}<m<\frac{2d}{d+2}$, we can choose $\theta>0$ sufficiently small so that \eqref{est-theta} holds. 
\end{proof}

\subsection{Scattering criterion}

Let us start with the following small data scattering which follows from the nonlinear estimates \eqref{non-est-1}--\eqref{non-est-3}. Since the proof is standard, we omit it. 

\begin{lemma}[Small data scattering] \label{lem-small-scat}
	Let $3\leq d\leq 5$, $0<\alpha<\min\left\{2,\frac{d}{2}\right\}$, $\frac{4-d}{2}<\alpha<\frac{6-d}{2}$, $\kappa>0$, and $\gamma \in \Rb$. Suppose that $(u,v)$ is a global $\Hc^1$-solution to \eqref{INLS} satisfying
	\begin{align} \label{E}
		\sup_{t\in \Rb} \|(u(t),v(t))\|_{\Hc^1} \leq E
	\end{align}
	for some constant $E>0$. Then there exists $\delta=\delta(E)>0$ sufficiently small such that if
	\begin{align} \label{small-cond-scat}
		\|(\Sc_1(t-T) u(T),\Sc_2(t-T) v(T))\|_{\Lc^{\overline{a}}([T,+\infty), \Lc^{\overline{r}}) \cap  \Lc^{\hat{a}}([T,+\infty), \Lc^{\hat{r}})} <\delta
	\end{align}
	for some $T>0$, then the solution scatters in $\Hc^1$ forward in time.	A similar statement holds for the negative time direction.
\end{lemma}

\begin{lemma}[Scattering criterion] \label{lem-scat-crite}
	Let $3\leq d\leq 5$, $0<\alpha<\min\left\{2,\frac{d}{2}\right\}$, $\frac{4-d}{2}<\alpha<\frac{6-d}{2}$, $\kappa>0$, and $\gamma \in \Rb$. Suppose that $(u,v)$ is a global $\Hc^1$-solution to \eqref{INLS} satisfying \eqref{E} for some constant $E>0$. Then there exist $\vareps=\vareps(E)>0$ and $R=R(E)>0$ such that if
	\begin{align} \label{scat-cond}
		\liminf_{t\rightarrow +\infty} \int_{|x| \leq R} |u(t,x)|^2 +2|v(t,x)|^2 \dd x \leq \vareps^2,
	\end{align}
	then the solution scatters in $\Hc^1$ forward in time. A similar statement holds for the negative time direction.
\end{lemma}

\begin{proof}
	We follow an idea of Murphy \cite{Murphy} (see also \cite{CC-PAMS}) which is inspired by earlier works \cite{DM-PAMS, Tao}. Thanks to Lemma \ref{lem-small-scat}, it suffices to justify the smallness condition \eqref{small-cond-scat}. Let $\vareps>0$ be a small constant. Let $T>0$ be a large time to be chosen later depending on $\vareps$. We will show that
	\begin{align} \label{est-T}
		\|(\Sc_1(t-T) u(T), \Sc_2(t-T) v(T))\|_{\Lc^{\overline{a}}([T,+\infty), \Lc^{\overline{r}}) \cap \Lc^{\hat{a}}([T,+\infty), \Lc^{\hat{r}})} \lesssim \vareps^\rho
	\end{align}
	for some constant $\rho>0$. To this end, we write, using the Duhamel formula, for $T>\vareps^{-\beta}$ with some $\beta>0$ to be determined later,
	\begin{align}
		(\Sc_1(t-T) u(T) &, \Sc_2(t-T) v(T)) \nonumber \\
		&= (\Sc_1(t) u_0, \Sc_2(t) v_0) + i \int_0^T \left(\Sc_1(t-\tau) |x|^{-\alpha} \overline{u}(\tau) v(\tau), \frac{1}{2}\Sc_2(t-\tau) |x|^{-\alpha} u^2(\tau)\right) \dd\tau \nonumber\\
		&=\Fc_0(t) + \Fc_1(t) + \Fc_2(t), \label{F012}
	\end{align}
	where $(u_0,v_0)$ is initial data at time $t=0$ and
	\begin{align*}
		\Fc_0(t) :&= (\Sc_1(t) u_0, \Sc_2(t) v_0), \\
		\Fc_1(t):&= i \int_{I_1}\left(\Sc_1(t-\tau) |x|^{-\alpha} \overline{u}(\tau) v(\tau), \frac{1}{2}\Sc_2(t-\tau) |x|^{-\alpha} u^2(\tau)\right) \dd\tau, \\
		\Fc_2(t):&=i \int_{I_2} \left(\Sc_1(t-\tau) |x|^{-\alpha} \overline{u}(\tau) v(\tau), \frac{1}{2}\Sc_2(t-\tau) |x|^{-\alpha} u^2(\tau)\right) \dd\tau,
	\end{align*}
	with $I_1=[0,T-\vareps^{-\beta}]$ and $I_2=[T-\vareps^{-\beta},T]$. The terms $\Fc_0, \Fc_1$, and $\Fc_2$ are referred as the linear, the distance past, and the recent past parts respectively.
	
	{\bf Step 1. The linear part.} To estimate the linear part, we use Strichartz estimates to have
	\[
	\|\Fc_0\|_{\Lc^{\overline{a}}([0,+\infty), \Lc^{\overline{r}}) \cap \Lc^{\hat{a}}([0,+\infty), \Lc^{\hat{r}})} \lesssim \|(u_0,v_0)\|_{\Hc^1}.
	\]
	By the monotone convergence theorem, there exists $T>\vareps^{-\beta}$ large so that
	\begin{align} \label{est-F0}
		\|\Fc_0\|_{\Lc^{\overline{a}}([T,+\infty), \Lc^{\overline{r}}) \cap \Lc^{\hat{a}}([T,+\infty), \Lc^{\hat{r}})} \lesssim \vareps.
	\end{align}
	
	{\bf Step 2. The distance past part.} We first observe that there exists $(\overline{b},\overline{e}) \in \Ac_0$ such that
	\[
	\frac{1}{\overline{a}} = \frac{1-s_c}{\overline{b}}+\theta s_c, \quad \frac{1}{\overline{r}} = \frac{1-s_c}{\overline{e}} + \frac{d-2-4\theta}{2d} s_c.
	\]
	In fact, as $(\overline{a}, \overline{r})\in \Ac_{s_c}$, we readily see that $\frac{2}{\overline{b}}+\frac{d}{\overline{e}} =\frac{d}{2}$. To ensure $(\overline{b}, \overline{e})\in \Ac_0$, it remains to check that $2<\overline{e}<\frac{2d}{d-2}$ or equivalently $\overline{b}>2$. The later is equivalent to 
	\[
	\frac{1-s_c}{\frac{1}{\overline{a}}-\theta s_c} >2 \Longleftrightarrow \frac{6-d-2\alpha}{\frac{6-d-2\alpha+\theta}{2(3-\theta)} - \theta(d-4+2\alpha)}>2.
	\]
	The limit as $\theta \rightarrow 0$ of the right hand side is strictly larger than 2, so the above condition is satisfied by taking $\theta>0$ small. By the H\"older  inequality, we have
	\[
	\|\Fc_1\|_{\Lc^{\overline{a}}([T,+\infty), \Lc^{\overline{r}})} \leq \|\Fc_1\|^{1-s_c}_{\Lc^{\overline{b}}([T,+\infty), \Lc^{\overline{e}})} \|\Fc_1\|^{s_c}_{\Lc^{\frac{1}{\theta}}([T,+\infty), \Lc^{\frac{2d}{d-2-4\theta}})}.
	\]
	Using the fact that
	\[
	\Fc_1(t) = (\Sc_1(t-T+\vareps^{-\beta}) u(T-\vareps^{-\beta}) -\Sc_1(t) u_0, \Sc_2(t-T+\vareps^{-\beta}) v(T-\vareps^{-\beta}) - \Sc_2(t) v_0),
	\]
	Strichartz estimates and \eqref{E} imply
	\[
	\|\Fc_1\|_{\Lc^{\overline{b}}([T,+\infty), \Lc^{\overline{e}})} \lesssim 1.
	\]
	On the other hand, by dispersive estimates \eqref{dis-est}, \eqref{non-est-4}, and \eqref{E}, we have for $t\geq T$,
	\begin{align*}
		\|\Fc_1(t)\|_{L^{\frac{2d}{d-2-4\theta}}_x} &\lesssim \int_0^{T-\vareps^{-\beta}} (t-\tau)^{-1-2\theta} \|(|x|^{-\alpha} \overline{u}v, |x|^{-\alpha} u^2)\|_{\Lc^{\frac{2d}{d+2+4\theta}}} \dd\tau \\
		&\lesssim \int_0^{T-\vareps^{-\beta}} (t-s)^{-1-2\theta} \|(u(\tau),v(\tau))\|^2_{\Hc^1} \dd\tau \\
		&\lesssim (t-T+\vareps^{-\beta})^{-2\theta}.
	\end{align*}
	It follows that
	\[
	\|\Fc_1\|_{\Lc^{\frac{1}{\theta}}([T,+\infty), \Lc^{\frac{2d}{d-2-4\theta}})} \lesssim \|(t-T+\vareps^{-\beta})^{-2\theta}\|_{L^{\frac{1}{\theta}}_t([T,+\infty))} \lesssim \vareps^{\beta \theta}.
	\]
	In particular, we get
	\[
	\|\Fc_1\|_{\Lc^{\overline{a}}([T,+\infty), \Lc^{\overline{r}})} \lesssim \vareps^{\beta \theta s_c}. 
	\]
	The estimate for $\|\Fc_1\|_{\Lc^{\hat{a}}([T,+\infty), \Lc^{\hat{r}})}$ is treated in a similar manner. In fact, we also observe that there exists $(\hat{b}, \hat{e}) \in \Ac_0$ such that
	\[
	\frac{1}{\hat{a}}=\frac{1-s_c}{\hat{b}} +\theta s_c, \quad \frac{1}{\hat{r}} = \frac{1-s_c}{\hat{e}} + \frac{d-2-4\theta}{2d} s_c.
	\] 
	Note that the condition $\hat{b}>2$ is equivalent to 
	\[
	\frac{1-s_c}{\frac{1}{\hat{a}} - \theta s_c} >2 \Longleftrightarrow \frac{6-d-2\alpha}{\frac{6-d-2\alpha}{3-\theta}-\theta(d-4+2\alpha)} >2
	\]
	which is clearly fulfilled for $\theta>0$ small. Estimating as above, we arrive at
	\[
	\|\Fc_1\|_{\Lc^{\hat{a}}([T,+\infty), \Lc^{\hat{r}})} \lesssim \vareps^{\beta \theta s_c}
	\]
	hence
	\begin{align} \label{est-F1}
		\|\Fc_1\|_{\Lc^{\overline{a}}([T,+\infty), \Lc^{\overline{r}}) \cap \Lc^{\hat{a}}([T,+\infty), \Lc^{\hat{r}})} \lesssim \vareps^{\beta \theta s_c}. 
	\end{align}
	
	{\bf Step 3. The recent past part.} Let $R>0$ be a large parameter depending on $\vareps$ to be determined shortly. Using \eqref{scat-cond} and enlarging $T$ if necessary, we have
	\[
	\int_{|x|\leq R} (|u(T,x)|^2 +2|v(T,x)|^2) \dd x \leq \vareps^2,
	\]
	hence
	\[
	\int_{\Rb^d} \varrho_R(x) (|u(T,x)|^2 +2|v(T,x)|^2) \dd x \leq \vareps^2,
	\]
	where $\varrho_R(x)=\varrho(x/R)$ with $C^\infty_0(\Rb^d)$ satisfying $0\leq \varrho \leq 1$ and
	\begin{align} \label{defi-varrho}
		\varrho(x) = \left\{
		\begin{array}{cl}
			1 &\text{if } |x| \leq 1/2, \\
			0 &\text{if } |x| \geq 1.
		\end{array}
		\right.
	\end{align}
	By Lemma \ref{lem-loca-viri-iden}, we have from \eqref{E} that
	\begin{align*}
		\left| \frac{\dd}{\dd t} \int_{\Rb^d} \varrho_R (|u(t)|^2 +2 |v(t)|^2) \dd x\right| &= \left|\ima \int_{\Rb^d} \nabla \varrho_R \cdot (\nabla u(t) \overline{u}(t) + \kappa \nabla v(t) \overline{v}(t)) \dd x \right| \\
		&\leq \|\nabla \varrho_R\|_{L^\infty_x} (\|\nabla u(t)\|_{L^2_x} \|u(t)\|_{L^2_x} + \kappa \|\nabla v(t)\|_{L^2_x} \|\nabla v(t)\|_{L^2_x}) \\
		&\lesssim R^{-1}.
	\end{align*}
	Thus we have for all $t\in I_2 = [T-\vareps^{-\beta}, T]$,
	\begin{align*}
		\int_{\Rb^d} \varrho_R(x)(|u(t,x)|^2+2|v(t,x)|^2) \dd x &=\int_{\Rb^d} \varrho_R(x)(|u(T,x)|^2+2|v(T,x)|^2)\dd x \\
		&\quad + \int_t^T \left( \frac{\dd}{\dd\tau} \int_{\Rb^d} \varrho_R(x) (|u(\tau,x)|^2 +2|v(\tau,x)|^2) \dd x\right) \dd\tau \\
		&\leq \vareps^2 + CR^{-1} (T-t) \\
		&\leq \vareps^2 + C R^{-1} \vareps^{-\beta} \\
		&\lesssim \vareps^2
	\end{align*}
	provided that $R\sim \vareps^{-2-\beta}$. In particular, we have
	\[
	\sup_{t\in I_2} \int_{\Rb^d} \varrho_R(x) (|u(t,x)|^2+2|v(t,x)|^2) \dd x \lesssim \vareps^2
	\]
	hence 
	\[
	\sup_{t\in I_2} \|(\varrho_R u(t), \varrho_Rv(t))\|_{\Lc^2} \lesssim \vareps^2
	\]
	as $\varrho_R^2\leq \varrho_R$. 
	
	On the other hand, we have from \eqref{est-A} and \eqref{E} that
	\begin{align*}
		\|\varrho_R |x|^{-\alpha} \overline{u}(t)v(t)\|_{L^{\hat{r}'}_x} &\leq \||x|^{-\alpha}\|_{L^\gamma_x(|x|\leq R)} \|u(t)\|^\theta_{L^{\overline{m}_+}_x} \|u(t)\|^{1-\theta}_{L^{\hat{r}}_x} \|\varrho_Rv(t)\|_{L^{\hat{r}}_x} \\
		&\lesssim R^{d\theta^2} \|u(t)\|_{H^1_x} \|\varrho_R v(t)\|_{L^{\hat{r}}_x} \\
		&\lesssim \vareps^{-d(2+\beta)\theta^2} \|u(t)\|_{H^1_x} \|\varrho_R v(t)\|_{L^2_x}^{\frac{2d-(d-2)\hat{r}}{2\hat{r}}} \|\varrho_R v(t)\|_{L^{\frac{2d}{d-2}}_x}^{\frac{d(\hat{r}-2)}{2\hat{r}}} \\
		&\lesssim \vareps^{-d(2+\beta)\theta^2} \|u(t)\|_{H^1_x} \|v(t)\|_{H^1_x}^{\frac{d(\hat{r}-2)}{2\hat{r}}}\|\varrho_R v(t)\|_{L^2_x}^{\frac{2d-(d-2)\hat{r}}{2\hat{r}}}  \\
		&\lesssim \vareps^{\frac{2d-(d-2)\hat{r}}{2\hat{r}}-d(2+\beta)\theta^2}, \quad \forall t\in I_2,
	\end{align*}
	where we recall that $\frac{1}{\gamma} = \frac{\alpha}{d} + \theta^2$ inside the ball. We also have from \eqref{est-A} and \eqref{E} that
	\begin{align*}
		\|(1-\varrho_R) |x|^{-\alpha} \overline{u}(t) v(t)\|_{L^{\hat{r}'}_x} &\lesssim \||x|^{-\alpha} \overline{u}(t) v(t)\|_{L^{\hat{r}'}_x(|x|\geq R/2)} \\
		&\leq \||x|^{-\alpha}\|_{L^\gamma_x(|x|\geq R/2)} \|u(t)\|^\theta_{L^{\overline{m}_-}_x} \|u(t)\|^{1-\theta}_{L^{\hat{r}}_x} \|v(t)\|_{L^{\hat{r}}_x} \\
		&\lesssim R^{-d\theta^2} \|u(t)\|_{H^1_x} \|v(t)\|_{H^1_x} \\
		&\lesssim R^{-d\theta^2} \\
		&\lesssim \vareps^{d(2+\beta)\theta^2}, \quad \forall t\in I_2,
	\end{align*}
	where $\frac{1}{\gamma}=\frac{\alpha}{d}-\theta^2$ outside the ball.  Thus we get
	\[
	\||x|^{-\alpha} \overline{u}(t) v(t)\|_{L^{\hat{r}'}_x} \lesssim \vareps^{\frac{2d-(d-2)\hat{r}}{2\hat{r}} - d(2+\beta)\theta^2} + \vareps^{d(2+\beta)\theta^2}, \quad \forall t\in I_2. 
	\]
	A similar estimate goes for $\||x|^{-\alpha} u^2(t)\|_{L^{\hat{r}'}_x}$ and we obtain
	\[
	\|(|x|^{-\alpha} \overline{u}v, |x|^{-\alpha} u^2)\|_{\Lc^\infty(I_2, \Lc^{\hat{r}'})} \lesssim \vareps^{\frac{2d-(d-2)\hat{r}}{2\hat{r}} - d(2+\beta)\theta^2} + \vareps^{d(2+\beta)\theta^2}.
	\]
	By Strichartz estimates, we have
	\begin{align*}
		\|\Fc_2\|_{\Lc^{\hat{a}}([T,+\infty), \Lc^{\hat{r}}) \cap \Lc^{\overline{a}}([T,+\infty), \Lc^{\overline{r}})} &\lesssim \|(|x|^{-\alpha} \overline{u}v, |x|^{-\alpha} u^2)\|_{\Lc^{\tilde{a}'}(I_2, \Lc^{\hat{r}'})} \\
		&\lesssim |I_2|^{\frac{1}{\tilde{a}'}} \|(|x|^{-\alpha} \overline{u}v, |x|^{-\alpha} u^2)\|_{\Lc^{\infty}(I_2, \Lc^{\hat{r}'})} \\
		&\lesssim \vareps^{\frac{2d-(d-2)\hat{r}}{2\hat{r}} - d(2+\beta)\theta^2-\frac{\beta}{\tilde{a}'}} + \vareps^{d(2+\beta)\theta^2-\frac{\beta}{\tilde{a}'}}.
	\end{align*}
	Now taking $\beta =\frac{2d\theta^2 \tilde{a}'}{2-d\theta^2 \tilde{a}'}$, we get
	\begin{align} \label{est-F2}
		\|\Fc_2\|_{\Lc^{\hat{a}}([T,+\infty), \Lc^{\hat{r}}) \cap \Lc^{\overline{a}}([T,+\infty), \Lc^{\overline{r}})} \lesssim \vareps^{\frac{2d-(d-2)\hat{r}}{2\hat{r}} -\frac{6d\theta^2}{2-d\theta^2\tilde{a}'}} + \vareps^{\frac{2d\theta^2\tilde{a}'}{2-d\theta^2 \tilde{a}'}}.
	\end{align}
	Collecting \eqref{F012}, \eqref{est-F0}, \eqref{est-F1}, and \eqref{est-F2}, we obtain
	\[
	\|(\Sc_1(t-T)u(T),\Sc_2(t-T)v(T))\|_{\Lc^{\hat{a}}([T,+\infty), \Lc^{\hat{r}}) \cap \Lc^{\overline{a}}([T,+\infty), \Lc^{\overline{r}})} \lesssim \vareps + \vareps^{\frac{2d\theta^3 \tilde{a}' s_c}{2-d\theta^2\tilde{a}'}} + \vareps^{\frac{2d-(d-2)\hat{r}}{2\hat{r}} -\frac{6d\theta^2}{2-d\theta^2\tilde{a}'}} + \vareps^{\frac{2d\theta^2\tilde{a}'}{2-d\theta^2 \tilde{a}'}}
	\]
	for some $T>\vareps^{-\frac{2d\theta^2\tilde{a}'}{2-d\theta^2\tilde{a}'}}$ which proves \eqref{est-T} by choosing $\theta>0$ sufficiently small. The proof is complete.
\end{proof}

\subsection{Energy scattering}
\begin{proposition}[Coercivity property] \label{prop-coer}
	Let $2\leq d\leq 5$, $0<\alpha<\min\{2,d\}$, $\frac{4-d}{2}<\alpha<\frac{6-d}{2}$, $\kappa>0$, and $\gamma \in \Rb$. Let $(u_0,v_0) \in \Hc^1$ satisfy \eqref{cond-gwp-supe-1} and \eqref{cond-gwp-supe-2}. Then there exists $R_0=R_0(\kappa, u_0,v_0,\varphi,\psi)>0$ such that the corresponding solution to \eqref{INLS} satisfies for all $R\geq R_0$,
	\begin{align} \label{coer-prop}
		\Gb(\varrho_Ru(t),\varrho_Rv(t)) \geq \delta \int_{\Rb^d} |x|^{-\alpha} (|\varrho_R(x)u(t,x)|^3+|\varrho_R(x)v(t,x)|^3) \dd x, \quad t\in \Rb
	\end{align}
	for some constant $\delta>0$ depending on $\kappa, u_0,v_0, \varphi$, and $\psi$, where $\varrho_R$ is as in \eqref{defi-varrho}.  
\end{proposition}
\begin{proof}
	The proof is divided into several steps. 
	
	{\bf Step 1. A uniform bound.} We first show that there exists $\rho=\rho(u_0,v_0,\varphi,\psi) \in (0,1)$ such that 
	\begin{align} \label{est-rho}
		\Kb(u(t),v(t)) (\Mb(u(t),v(t)))^\sigma \leq (1-\rho) \Kb(\varphi,\psi) (\Mb(\varphi,\psi))^\sigma, \quad \forall t\in \Rb.
	\end{align}
	We first observe that for $(u_0,v_0) \ne (0,0)$ satisfying \eqref{cond-gwp-supe-1} and \eqref{cond-gwp-supe-2}, we have $\Hb(u_0,v_0) >0$. To see this, let $g$ be as in \eqref{defi-g}. Using \eqref{iden-var-psi}, we observe that
	\begin{align} \label{C-GN}
	C_{\GN} = \frac{\Pb(\varphi,\psi)}{\left(\Kb(\varphi,\psi)\right)^{\frac{d+2\alpha}{4}} \left(\Mb(\varphi,\psi)\right)^{\frac{6-d-2\alpha}{4}}}= \frac{2}{d+2\alpha} \left(\Kb(\varphi,\psi) \left(\Mb(\varphi,\psi)\right)^\sigma\right)^{-\frac{d+2\alpha-4}{4}}.
	\end{align}
	Thus $g'(\lambda_0)=0$ with
	\[
	\lambda_0 = \left(\frac{2}{(d+2\alpha)C_{\GN}}\right)^{\frac{4}{d+2\alpha-4}} = \Kb(\varphi,\psi)(\Mb(\varphi,\psi))^\sigma
	\]
	and $g$ is strictly increasing on $(0,\lambda_0)$. This together with $(u_0,v_0)\ne(0,0)$ and \eqref{cond-gwp-supe-2} yield
	\[
	\Hb(u_0,v_0) (\Mb(u_0,v_0))^\sigma \geq g\left(\Kb(u_0,v_0)(\Mb(u_0,v_0))^\sigma\right) >g(0)=0 
	\]
	which implies $\Hb(u_0,v_0)>0$. 
	
	Now, using \eqref{cond-gwp-supe-1}, we take $\vartheta=\vartheta(u_0,v_0,\varphi,\psi) \in (0,1)$ so that
	\[
	\Hb(u_0,v_0) (\Mb(u_0,v_0))^\sigma \leq (1-\vartheta) \Eb_0(\varphi,\psi) (\Mb(\varphi,\psi))^\sigma.
	\]
	Since (see the proof of Proposition \ref{prop-gwp} in the super-critical case)
	\[
	g\left(\Kb(u(t),v(t)) (\Mb(u(t),v(t)))^\sigma \right) \leq \Hb(u_0,v_0) (\Mb(u_0,v_0))^\sigma, \quad \forall t\in \Rb,
	\] 
	and
	\[
	\Eb_0(\varphi,\psi) (\Mb(\varphi,\psi))^\sigma = \frac{d+2\alpha-4}{2(d+2\alpha)} \Kb(\varphi,\psi) (\Mb(\varphi,\psi))^\sigma = \frac{d+2\alpha-4}{4} C_{\GN} \left( \Kb(\varphi,\psi) (\Mb(\varphi,\psi))^\sigma\right)^{\frac{d+2\alpha}{4}},
	\]
	we get
	\begin{multline*}
		\frac{d+2\alpha}{d+2\alpha-4} \frac{\Kb(u(t),v(t)) (\Mb(u(t),v(t)))^\sigma}{\Kb(\varphi,\psi) (\Mb(\varphi,\psi))^\sigma} \\
		- \frac{4}{d+2\alpha-4} \left(\frac{\Kb(u(t),v(t)) (\Mb(u(t),v(t)))^\sigma}{\Kb(\varphi,\psi) (\Mb(\varphi,\psi))^\sigma}  \right)^{\frac{d+2\alpha}{4}} \leq 1-\vartheta, \quad \forall t\in \Rb.
	\end{multline*}
	Denote
	\[
	\lambda(t):= \frac{\Kb(u(t),v(t)) (\Mb(u(t),v(t)))^\sigma}{\Kb(\varphi,\psi) (\Mb(\varphi,\psi))^\sigma}.
	\]
	We see that $\lambda(t)\in (0,1)$ for all $t\in \Rb$ due to \eqref{claim-gwp-supe} and
	\[
	h(\lambda(t)) \leq 1-\vartheta, \quad \forall t\in \Rb,
	\]
	where 
	\[
	h(\lambda):= \frac{d+2\alpha}{d+2\alpha-4}\lambda - \frac{4}{d+2\alpha-4} \lambda^{\frac{d+2\alpha}{4}}.
	\]
	As $h$ is continuous on $[0,1]$, $h(0)=0$, and $h(1)=1$, there exists $\rho\in (0,1)$ such that $h(1-\rho) = 1-\vartheta$. In particular, we have
	\[
	h(\lambda(t)) \leq h(1-\rho), \quad \forall t\in \Rb.
	\]
	Since $h$ is strictly increasing on $(0,1)$, we infer that $\lambda(t) \leq 1-\rho$ for all $t\in \Rb$ and \eqref{est-rho} follows.
	
	{\bf Step 2. A truncated uniform bound.} There exists $R_0=R_0(\kappa, u_0,v_0,\varphi,\psi)>0$ such that for all $R\geq R_0$,
	\begin{align} \label{est-varrho-R}
		\Kb(\varrho_Ru(t),\varrho_Rv(t)) (\Mb(\varrho_Ru(t),\varrho_Rv(t)))^\sigma \leq \left(1-\frac{\rho}{2}\right) \Kb(\varphi,\psi) (\Mb(\varphi,\psi))^{\sigma}, \quad \forall t\in \Rb. 
	\end{align}
	In fact, as $0\leq \varrho_R \leq 1$, we have $\Mb(\varrho_Ru(t),\varrho_Rv(t)) \leq \Mb(u(t),v(t))$ for all $R>0$ and all $t\in \Rb$. On the other hand, by integration by parts, we have
	\[
	\int_{\Rb^d} |\nabla(\varrho f)|^2 \dd x = \int_{\Rb^d} \varrho^2 |\nabla f|^2 \dd x - \int_{\Rb^d} \varrho \Delta \varrho |f|^2 \dd x
	\]
	which implies
	\begin{align*}
		\Kb(\varrho_Ru(t),\varrho_Rv(t)) &= \int_{\Rb^d} \varrho_R^2 (|\nabla u(t)|^2 + \kappa |\nabla v(t)|^2) \dd x - \int_{\Rb^d} \varrho_R \Delta \varrho_R (|u(t)|^2 + \kappa |v(t)|^2) \dd x \\
		&\leq \Kb(u(t),v(t)) + C(\kappa) R^{-2} \Mb(u(t),v(t)).
	\end{align*}
	By the conservation of mass and \eqref{est-rho}, we get
	\begin{align*}
		\Kb(\varrho_Ru(t), \varrho_Rv(t)) (\Mb(\varrho_Ru(t),\varrho_Rv(t)))^\sigma &\leq \Kb(u(t),v(t)) (\Mb(u(t),v(t)))^\sigma + C(\kappa) R^{-2} (\Mb(u(t),v(t)))^{\sigma+1} \\
		&\leq (1-\rho) \Kb(\varphi,\psi) (\Mb(\varphi,\psi))^\sigma + C(\kappa) R^{-2} (\Mb(u_0,v_0))^{\sigma+1} \\
		&\leq \left(1-\frac{\rho}{2}\right) \Kb(\varphi,\psi) (\Mb(\varphi,\psi))^\sigma, \quad \forall R\geq R_0, \quad \forall t\in \Rb
	\end{align*}
	provided that $R_0>0$ is taken sufficiently large depending on $\rho$, $\kappa$, and $\Mb(u_0,v_0)$. This shows \eqref{est-varrho-R}.
	
	{\bf Step 3. A coercivity estimate.} We claim that if
	\[
	\Kb(f,g) (\Mb(f,g))^\sigma \leq (1-\nu) \Kb(\varphi,\psi)(\Mb(\varphi,\psi))^\sigma
	\]
	for some $0<\nu <1$, then there exists $\delta =\delta(\nu, \kappa, \varphi, \psi)>0$ such that
	\[
	\Gb(f,g) \geq \delta\int_{\Rb^d} |x|^{-\alpha}(|f|^3+|g|^3) \dd x.
	\]
	Thanks to this claim, the desired estimate \eqref{coer-prop} follows immediately from \eqref{est-varrho-R}. To prove the claim, we write
	\[
	\Gb(f,g) = \Kb(f,g) - \frac{d+2\alpha}{2} \Pb(f,g) = \frac{d+2\alpha}{2} \Eb_0(f,g) - \frac{d+2\alpha-4}{4} \Kb(f,g).
	\]
	Using the Gagliardo--Nirenberg inequality \eqref{GN-ineq} and \eqref{C-GN}, we have
	\begin{align*}
		\Eb_0(f,g) &\geq \frac{1}{2} \Kb(f,g) - C_{\GN} (\Kb(f,g))^{\frac{d+2\alpha}{4}} (\Mb(f,g))^{\frac{6-d-2\alpha}{4}} \\
		&= \frac{1}{2} \Kb(f,g) \left(1-C_{\GN} \left(\Kb(f,g) (\Mb(f,g))^\sigma\right)^{\frac{d+2\alpha-4}{4}} \right) \\
		&\geq \frac{1}{2} \Kb(f,g) \left(1-C_{\GN} \left( (1-\nu) \Kb(\varphi,\psi) (\Mb(\varphi,\psi))^\sigma\right)^{\frac{d+2\alpha-4}{4}}\right) \\
		&= \frac{1}{2}\Kb(f,g) \left(1-\frac{2}{d+2\alpha} (1-\nu)^{\frac{d+2\alpha-4}{4}}\right).
	\end{align*}
	It follows that
	\begin{align}
		\Gb(f,g) &\geq \frac{d+2\alpha}{4} \Kb(f,g) \left(1-\frac{2}{d+2\alpha}(1-\nu)^{\frac{d+2\alpha-4}{4}}\right) - \frac{d+2\alpha-4}{4} \Kb(f,g) \nonumber \\
		&= \left(1-(1-\nu)^{\frac{d+2\alpha-4}{4}}\right)\Kb(f,g). \label{coer-proof-1}
	\end{align}
	On the other hand, we have from the standard Gagliardo--Nirenberg inequality \eqref{gn-ineq} that
	\begin{align*}
		\int_{\Rb^d} |x|^{-\alpha}(|f|^3+|g|^3) \dd x &\leq C_{\gn} \left(\|\nabla f\|^{\frac{d+2\alpha}{2}}_{L^2} \|f\|^{\frac{6-d-2\alpha}{2}}_{L^2} + \|\nabla g\|^{\frac{d+2\alpha}{2}}_{L^2} \|g\|^{\frac{6-d-2\alpha}{2}}_{L^2} \right) \\
		&\leq C(\kappa) (\Kb(f,g))^{\frac{d+2\alpha}{4}} (\Mb(f,g))^{\frac{6-d-2\alpha}{4}} \\
		& = C(\kappa) \Kb(f,g) \left( \Kb(f,g) (\Mb(f,g))^\sigma\right)^{\frac{d+2\alpha-4}{4}} \\
		&\leq C(\kappa)\left((1-\nu) \Kb(\varphi,\psi) (\Mb(\varphi,\psi))^\sigma \right)^{\frac{d+2\alpha-4}{4}} \Kb(f,g)
	\end{align*}
	which together with \eqref{coer-proof-1} yields
	\[
	\int_{\Rb^d} |x|^{-\alpha}(|f|^3+|g|^3) \dd x \leq \frac{C(\kappa)\left((1-\nu) \Kb(\varphi,\psi) (\Mb(\varphi,\psi))^\sigma \right)^{\frac{d+2\alpha-4}{4}}}{\left(1-(1-\nu)^{\frac{d+2\alpha-4}{4}}\right)} \Gb(f,g).
	\]
	This proves the claim and the proof is complete.
\end{proof}

\begin{proof}[Proof of Theorem \ref{theo-scat}]
	It suffices to consider the positive time direction since the one in the negative time direction is treated in a similar manner. The proof is divided  into two steps.
	
	{\bf Step 1. A space-time estimate.} We first show that there exists $C=C(\kappa, u_0,v_0,\varphi,\psi)>0$ such that for any time interval $J \subset \Rb$,
	\begin{align} \label{space-time-est}
		\int_J \int_{\Rb^d} |x|^{-\alpha}(|u(t,x)|^3 + |v(t,x)|^3) \dd x \dd t \leq C|J|^{\frac{1}{1+\alpha}}.
	\end{align}
	To see this, we introduce a cutoff function $\eta:[0,\infty) \rightarrow [0,2]$ satisfying
	\[
	\eta(r) = \left\{
	\begin{array}{ccl}
		2 &\text{if}& 0\leq r\leq 1, \\
		0 &\text{if}& r\geq 2.
	\end{array}
	\right.
	\]
	Define the function $\phi:[0,\infty) \rightarrow [0, \infty)$ by
	\[
	\phi(r):=\int_0^r \int_0^s \eta(\tau) \dd\tau.
	\]
	Let $R>0$. We define  the radial function 
	\begin{align} \label{phi-R}
		\phi_R(x)=\phi_R(r) = R^2 \phi(r/R), \quad r=|x| 
	\end{align}
	and the localized virial quantity
	\begin{align} \label{M-phi-R}
		\Mcal_{\phi_R}(t) = \ima \int_{\Rb^d} \nabla \phi_R \cdot (\nabla u(t) \overline{u}(t)+\nabla v(t) \overline{v}(t)) \dd x.
	\end{align}
	Using Remark \ref{rem-loca-viri-rad}, we have
	\begin{align*}
		\frac{\dd}{\dd t} \Mcal_{\phi_R}(t) &= \frac{1}{4} \int_{\Rb^d} \Delta^2\phi_R (|u(t)|^2 + \kappa |v(t)|^2) \dd x + \int_{\Rb^d} \frac{\phi'_R}{r}(|\nabla u(t)|^2 +\kappa|\nabla v(t)|^2) \dd x \\
		&\quad + \int_{\Rb^d} \left(\frac{\phi''_R}{r^2}-\frac{\phi'_R}{r^3}\right) (|x\cdot \nabla u(t)|^2 + \kappa |x\cdot \nabla v(t)|^2) \dd x \\
		&\quad -\frac{1}{2} \rea \int_{\Rb^d} \left(\phi''_R + (d-1+2\alpha)\frac{\phi'_R}{r}\right) |x|^{-\alpha} u^2(t) \overline{v}(t) \dd x.
	\end{align*}
	As $\phi_R(r)=r^2$ for $0\leq r\leq R$, we rewrite the above identity as
	\begin{align*}
		\frac{\dd}{\dd t} \Mcal_{\phi_R}(t) &= 2 \int_{|x|\leq R} (|\nabla u(t)|^2 +\kappa |\nabla v(t)|^2) \dd x - (d+2\alpha) \rea \int_{|x|\leq R} |x|^{-\alpha} u^2(t) \overline{v}(t) \dd x \\
		&\quad +E_1(u(t),v(t)) + E_2(u(t),v(t)) + E_3(u(t),v(t)),
	\end{align*}
	where
	\begin{align*}
		E_1(u,v) &=\frac{1}{4} \int_{\Rb^d} \Delta^2\phi_R (|u|^2 + \kappa |v|^2) \dd x, \\ 
		E_2(u,v) &=\int_{|x|>R} \frac{\phi'_R}{r}(|\nabla u|^2 +\kappa|\nabla v|^2) \dd x + \int_{|x|>R} \left(\frac{\phi''_R}{r^2}-\frac{\phi'_R}{r^3}\right) (|x\cdot \nabla u|^2 + \kappa |x\cdot \nabla v|^2) \dd x, \\
		E_3(u,v) &= -\frac{1}{2} \rea \int_{|x|>R} \left(\phi''_R + (d-1+2\alpha)\frac{\phi'_R}{r}\right) |x|^{-\alpha} u^2 \overline{v} \dd x.
	\end{align*}
	By the conservation of mass, we have
	\[
	|E_1(u(t),v(t))| \leq C(\kappa, \Mb(u_0,v_0)) R^{-2}.
	\]
	As $0\leq \phi''_R, \frac{\phi'_R}{r} \leq 2$, we have
	\begin{align*}
		\frac{\phi'_R}{r} |\nabla u|^2 + \left(\frac{\phi''_R}{r^2}-\frac{\phi'_R}{r^3}\right) |x\cdot \nabla u|^2 &= \frac{\phi''_R}{r^2}|x\cdot \nabla u|^2 + \frac{\phi'_R}{r} \left(|\nabla u|^2-\frac{1}{r^2}|x\cdot \nabla u|^2\right) \\
		&\geq \frac{\phi''_R}{r^2}|x\cdot \nabla u|^2 \geq 0.
	\end{align*}
	A similar estimate holds for $v$ and we get $E_2(u(t),v(t))\geq 0$. We also have 
	\begin{align*}
		|E_3(u(t),v(t))| &\leq C \int_{|x|>R} |x|^{-\alpha} |u(t)|^2 |v(t)| \dd x \\
		&\leq CR^{-\alpha} \|u(t)\|^2_{L^3} \|v(t)\|_{L^3} \\
		&\leq CR^{-\alpha} \|u(t)\|^2_{H^1} \|v(t)\|_{H^1} \\
		&\leq CR^{-\alpha}
	\end{align*}
	for some constant $C>0$ depending on $\kappa, u_0,v_0, \varphi$, and $\psi$. Here we have used the Sobolev embedding $H^1(\Rb^d) \subset L^3(\Rb^d)$ for $3\leq d\leq 5$ and \eqref{blow-alte}. 
	
	It follows that
	\begin{align*}
		\frac{\dd}{\dd t} \Mcal_{\phi_R}(t) &\geq 2 \int_{|x|\leq R} (|\nabla u(t)|^2 + \kappa |\nabla v(t)|^2) \dd x - (d+2\alpha) \rea \int_{|x|\leq R} |x|^{-2} u^2(t) \overline{v}(t) \dd x + CR^{-\alpha}, \quad \forall t\in \Rb
	\end{align*}
	as $\alpha<2$. Now let $\varrho_R$ be as in \eqref{defi-varrho}. Observe that
	\begin{align*}
		\int_{\Rb^d} |\nabla(\varrho_R u(t))|^2 \dd x &=\int_{\Rb^d} \varrho^2_R|\nabla u(t)|^2 \dd x - \int_{\Rb^d} \varrho_R \Delta \varrho_R |u(t)|^2 \dd x \\
		&=\int_{|x|\leq R} |\nabla u(t)|^2 \dd x -\int_{R/2<|x|\leq R} (1-\varrho^2_R) |\nabla u(t)|^2 \dd x -\int_{\Rb^d} \varrho_R \Delta \varrho_R |u(t)|^2 \dd x \\
		&\leq \int_{|x|\leq R} |\nabla u(t)|^2 \dd x -\int_{\Rb^d} \varrho_R\Delta \varrho_R |u(t)|^2 \dd x
	\end{align*}
	due to $0\leq \varrho_R \leq 1$. The same estimate goes for $v$ and we get
	\begin{align*}
		\int_{|x|\leq R} (|\nabla u(t)|^2 + \kappa |\nabla v(t)|^2) \dd x &\geq \int_{\Rb^d} (|\nabla (\varrho_R u(t))|^2 + \kappa |\nabla(\varrho_R v(t))|^2) \dd x + \int_{\Rb^d} \varrho_R \Delta \varrho_R (|u(t)|^2 +\kappa |v(t)|^2) \dd x \\
		&= \Kb(\varrho_Ru(t), \varrho_Rv(t)) + O(R^{-2}),
	\end{align*}
	where we have used $\|\Delta \varrho_R\|_{L^\infty} \lesssim R^{-2}$ and the mass conservation to get the second line. In addition, we have
	\begin{align*}
		\rea \int_{\Rb^d} |x|^{-\alpha} (\varrho_R u(t))^2 \overline{\varrho_R v(t)} \dd x &= \rea \int_{|x|\leq R} |x|^{-\alpha} u^2(t) \overline{v}(t) \dd x - \rea \int_{R/2 <|x| \leq R} |x|^{-\alpha} (1-\varrho_R^3) u^2(t) \overline{v}(t) \dd x.
	\end{align*}
	The last term is estimated as for $E_3(u,v)$ and we obtain
	\[
	\rea \int_{|x|\leq R} |x|^{-\alpha} u^2(t) \overline{v}(t) \dd x = \Pb(\varrho_R u(t),\varrho_Rv(t)) + O(R^{-\alpha}).
	\]
	In particular, using \eqref{coer-prop}, there exists $R_0=R_0(\kappa, u_0,v_0,\varphi,\psi)>0$ such that for all $R\geq R_0$,
	\begin{align*}
		2 \int_{|x|\leq R} (|\nabla u(t)|^2 + \kappa |\nabla v(t)|^2) \dd x &- (d+2\alpha) \rea \int_{|x|\leq R} |x|^{-2} u^2(t) \overline{v}(t) \dd x \\
		&\geq 2\Kb(\varrho_Ru(t), \varrho_Rv(t)) -(d+2\alpha) \Pb(\varrho_Ru(t),\varrho_Rv(t)) + O(R^{-\alpha}) \\
		&= 2\Gb(\varrho_Ru(t), \varrho_Rv(t)) + O(R^{-\alpha}) \\
		&\geq 2\delta \int_{\Rb^d} |x|^{-\alpha} (|\varrho_R u(t)|^3+|\varrho_Rv(t)|^3) \dd x + O(R^{-\alpha}), \quad \forall t\in \Rb.
	\end{align*}
	We have proved that for all $R\geq R_0$,
	\[
	\frac{\dd}{\dd t} \Mcal_{\phi_R}(t) \geq 2\delta \int_{\Rb^d} |x|^{-\alpha} (|\varrho_R u(t)|^3+ |\varrho_R v(t)|^3) \dd x + O(R^{-\alpha}), \quad \forall t\in \Rb. 
	\]
	Integrating in time, we get for any time interval $J \subset \Rb$,
	\[
	\int_J \int_{\Rb^d} |x|^{-\alpha} (|\varrho_R u(t)|^3+ |\varrho_R v(t)|^3) \dd x \dd t \lesssim \sup_{t\in J} |\Mcal_{\phi_R}(t)| + R^{-\alpha} |J|. 
	\]
	By the definition of $\varrho_R$, we get
	\[
	\int_J \int_{|x|\leq R/2} |x|^{-\alpha} (| u(t)|^3+ |v(t)|^3) \dd x \dd t \lesssim \sup_{t\in J} |\Mcal_{\phi_R}(t)| + R^{-2} |J| \lesssim R + R^{-\alpha} |J|,
	\]
	where we use that
	\[
	|\Mcal_{\phi_R}(t)| \leq \|\nabla \phi_R\|_{L^\infty} (\|\nabla u(t)\|_{L^2} \|v(t)\|_{L^2} + \|\nabla v(t)\|_{L^2} \|v(t)\|_{L^2}) \lesssim R.
	\]
	On the other hand, estimating as for $E_3(u,v)$, we have
	\[
	\int_{|x|>R/2} |x|^{-\alpha} (|u(t)|^3+|v(t)|^3) \dd x \lesssim R^{-\alpha}
	\]
	which implies
	\[
	\int_J \int_{|x|>R/2} |x|^{-\alpha} (|u(t)|^3+|v(t)|^3) \dd x \dd t \lesssim R^{-\alpha} |J|.
	\]
	Collecting these estimates, we get for all $R\geq R_0$,
	\[
	\int_J \int_{\Rb^d} |x|^{-\alpha} (|u(t)|^3+|v(t)|^3) \dd x \dd t \lesssim R + R^{-\alpha}|J|.
	\]
	For $|J|\geq R_0^{1+\alpha}$, we take $R=|J|^{\frac{1}{1+\alpha}}\geq R_0$ to get
	\[
	\int_J \int_{\Rb^d} |x|^{-\alpha} (|u(t)|^3+|v(t)|^3) \dd x \dd t \lesssim |J|^{\frac{1}{1+\alpha}}.
	\]
	If $|J| \leq R_0^{1+\alpha}$, we use the standard Gagliardo--Nirenberg inequality \eqref{gn-ineq} and \eqref{blow-alte} to get
	\begin{align*}
		\int_J \int_{\Rb^d} |x|^{-\alpha} (|u(t)|^3+|v(t)|^3) \dd x \dd t &\leq C(\kappa)\int_J (\Kb(u(t),v(t))^{\frac{d+2\alpha}{4}} (\Mb(u(t),v(t)))^{\frac{6-d-2\alpha}{4}} \dd t \\
		&\lesssim |J| \lesssim R_0^\alpha |J|^{\frac{1}{1+\alpha}}. 
	\end{align*}
	In all cases, we have proved \eqref{space-time-est}.

	{\bf Step 2. A $L^2$-limit.} From \eqref{space-time-est}, we infer that there exists $t_n\rightarrow +\infty$ such that for any $R>0$,
	\begin{align} \label{L2-limi}
		\lim_{n\rightarrow \infty} \int_{|x|\leq R} |u(t_n,x)|^2 + 2 |v(t_n,x)|^2 \dd x =0.
	\end{align}
	This condition, combined with the scattering criterion given in Lemma \ref{lem-scat-crite}, yields the energy scattering in the positive time direction. To see \eqref{L2-limi}, we infer from \eqref{space-time-est} that 
	\[
	\liminf_{t\rightarrow +\infty} \int_{\Rb^d} |x|^{-\alpha}(|u(t)|^3+|v(t)|^3)\dd x =0.
	\]
	Indeed, if it is not true, then there exist $t_0>0$ and $\rho_0>0$ such that 
	\[
	\int_{\Rb^d} |x|^{-\alpha} (|u(t)|^3+|v(t)|^3)\dd x \geq \rho_0, \quad \forall t\geq t_0. 
	\]
	Take an interval $J\subset [t_0,+\infty)$. We have
	\[
	\int_J \int_{\Rb^d}|x|^{-\alpha} (|u(t)|^3+|v(t)|^3)\dd x \dd t \geq \rho_0 |J|
	\]
	which contradicts \eqref{space-time-est} for $|J|$ sufficiently large. Thus there exists $t_n \rightarrow +\infty$ such that
	\[
	\lim_{n\rightarrow \infty} \int_{\Rb^d} |x|^{-\alpha}(|u(t_n)|^3+|v(t_n)|^3)\dd x =0.
	\]
	Now let $R>0$. We have
	\begin{align*}
		\int_{|x|\leq R} |u(t_n)|^2 \dd x &\leq \left(\int_{|x|\leq R} |x|^{2\alpha} \dd x\right)^{\frac{1}{3}} \left(\int_{\Rb^d} |x|^{-\alpha} |u(t_n)|^3 \dd x\right)^{\frac{2}{3}} \\
		&\leq C R^{\frac{d+2\alpha}{3}} \left(\int_{\Rb^d} |x|^{-\alpha} |u(t_n)|^3 \dd x\right)^{\frac{2}{3}}.
	\end{align*}
	Thus
	\begin{align*}
		\int_{|x|\leq R} |u(t_n)|^2 +2 |v(t_n)|^2 \dd x &\leq C R^{\frac{d+2\alpha}{3}}\left[\left(\int_{\Rb^d} |x|^{-\alpha} |u(t_n)|^3 \dd x\right)^{\frac{2}{3}} + \left(\int_{\Rb^d} |x|^{-\alpha} |v(t_n)|^3 \dd x\right)^{\frac{2}{3}}\right]  \\
		&\leq C R^{\frac{d+2\alpha}{3}} \left(\int_{\Rb^d} |x|^{-\alpha} (|u(t_n)|^3 +|v(t_n)|^3) \dd x\right)^{\frac{2}{3}} \rightarrow 0 \text{ as } n\rightarrow \infty
	\end{align*}
	which proves \eqref{L2-limi}. The proof is complete.
\end{proof}

\section{Blow-up}
\label{S-blow}
\setcounter{equation}{0}

\subsection{Virial identities}

\begin{lemma} \label{lem-loca-viri-iden}
	Let $(u,v)$ be a $\Hc^1$-solution to \eqref{INLS} defined on the maximal time interval $(-T_*,T^*)$. Let $\chi: \Rb^d \rightarrow \Rb$ be a sufficiently smooth and decaying function. Define
	\begin{align*} 
	\Vc_{\chi}(t):= \int_{\Rb^d} \chi (|u(t)|^2+2|v(t)|^2) \dd x
	\end{align*}
	and
	\begin{align*}
	\Mcal_{\chi}(t) := \ima \int_{\Rb^d} \nabla \chi \cdot (\nabla u(t) \overline{u}(t) + \nabla v(t) \overline{v}(t)) \dd x. 
	\end{align*}
	Then for all $t\in (-T_*,T^*)$, 
	\begin{align*} 
	\frac{\dd}{\dd t} \Vc_{\chi}(t) = \ima \int_{\Rb^d} \nabla \chi \cdot (\nabla u(t) \overline{u}(t) + 2 \kappa \nabla v(t) \overline{v}(t)) \dd x 
	\end{align*}
	and 
	\begin{align*}
	\frac{\dd}{\dd t} \Mcal_{\chi}(t) &= \frac{1}{4} \int_{\Rb^d} \Delta^2 \chi (|u(t)|^2 + \kappa |v(t)|^2) \dd x +\sum_{j,k=1}^d \rea \int_{\Rb^d} \partial^2_{jk} \chi (\partial_j \overline{u}(t) \partial_k u(t) + \kappa \partial_j \overline{v}(t) \partial_k v(t)) \dd x \\
	&\quad- \frac{1}{2} \rea \int_{\Rb^d} \Delta \chi |x|^{-\alpha} u^2(t) \overline{v}(t) \dd x + \rea \int_{\Rb^d} \nabla \chi \cdot \nabla (|x|^{-\alpha}) u^2(t) \overline{v}(t) \dd x. 
	\end{align*}
	In particular, if $\kappa=\frac{1}{2}$, then we have
	\begin{align} \label{viri-mass-reso}
	\frac{\dd}{\dd t} V_{\chi}(t) = \Mcal_{\chi}(t).
	\end{align}
\end{lemma}

\begin{remark}\label{x2-chi}
	If $\chi(x)=|x|^2$, we have
	\[
	\frac{\dd}{\dd t} \Mcal_{|x|^2}(t) = 2 \Gb(u(t),v(t)),
	\]
	where $G(u,v)$ is as in \eqref{Gb}.	
\end{remark}

\begin{remark} \label{rem-loca-viri-rad}
	If $\chi$ is radial, then we have
	\begin{align*}
	\frac{\dd}{\dd t} \Mcal_{\chi}(t) &= \frac{1}{4} \int_{\Rb^d} \Delta^2\chi (|u(t)|^2 + \kappa |v(t)|^2) \dd x + \int_{\Rb^d} \frac{\chi'}{r} (|\nabla u(t)|^2 + \kappa |\nabla v(t)|^2) \dd x \\
	&\quad + \int_{\Rb^d} \left(\frac{\chi''}{r^2}-\frac{\chi'}{r^3}\right) (|x\cdot \nabla u(t)|^2 + \kappa |x\cdot \nabla v(t)|^2) \dd x \\
	&\quad -\frac{1}{2} \rea \int_{\Rb^d} \left( \chi'' + (d-1+2\alpha)\frac{\chi'}{r}\right) |x|^{-\alpha} u^2(t) \overline{v}(t) \dd x.
	\end{align*}
	This follows from a direct computation using the fact that
	\begin{align*}
	\sum_{j,k}^n \partial^2_{jk} \chi \partial_j \overline{u} \partial_k u &= \frac{\chi'}{r} |\nabla u|^2 + \left(\frac{\chi''}{r^2}-\frac{\chi'}{r^3}\right) |x\cdot \nabla u|^2, \\
	\Delta \chi &= \chi'' + (n-1)\frac{\chi'}{r}, \\
	\nabla \chi \cdot \nabla (|x|^{-\alpha})&= -\alpha |x|^{-\alpha} \frac{\chi'}{r}.
	\end{align*}
\end{remark}

\begin{proof}[Proof of Lemma \ref{lem-loca-viri-iden}]
	It follows directly from the following observation: for any $a\in \Rb$,
		\begin{align*}
		\partial_t (|u|^2 + a |v|^2) &= -\nabla \cdot \ima (\overline{u} \nabla u) - a\kappa \nabla \cdot \ima (\overline{v} \nabla v) - (a-2) \ima (|x|^{-\alpha} u^2 \overline{v}), \\
		\partial_t (\ima (\overline{u} \partial_k u) +  \ima(\overline{v} \partial_k v)) &= \frac{1}{4} \partial_k \Delta(|u|^2) + \frac{\kappa}{4} \partial_k \Delta(|v|^2) - \sum_{j=1}^d \partial_j \rea (\partial_j \overline{u} \partial_k u) -  \kappa \sum_{j=1}^d \partial_j \rea (\partial_j \overline{v} \partial_k v) \\
		&\quad + \frac{1}{2} \partial_k \rea (|x|^{-\alpha} u^2 \overline{v}) + \rea (\partial_k(|x|^{-\alpha}) u^2 \overline{v}), \quad \forall k=1, \cdots, d.
		\end{align*}
		To see these identities, we apply Lemma \ref{lem-viri-iden} with $\beta =\frac{1}{2}$ and $H = -|x|^{-\alpha} \overline{u} v$ to get
		\begin{align*}
		\partial_t |u|^2 &= -\nabla \cdot \ima (\overline{u} \nabla u) - 2\ima (|x|^{-\alpha} \overline{u}^2 v) = -\nabla \cdot \ima (\overline{u} \nabla u) + 2\ima (|x|^{-\alpha} u^2\overline{v}), \\
		\partial_t \ima (\overline{u} \partial_k u) &= \frac{1}{4} \partial_k\Delta(|u|^2) - \sum_{j=1}^d \partial_j \rea (\partial_j \overline{u} \partial_k u) -2 \rea (|x|^{-\alpha} u \partial_k u \overline{v}) + \partial_k \rea (|x|^{-\alpha} u^2 \overline{v}).
		\end{align*}
		Applying Lemma \ref{lem-viri-iden} for $\beta = \frac{\kappa}{2}$ and $H=\gamma v - \frac{1}{2} |x|^{-\alpha} u^2$, we get
		\begin{align*}
		\partial_t |v|^2 &= -\kappa \nabla \cdot \ima(\overline{v} \nabla v) + 2\ima \left[\overline{v}\left(\gamma v-\frac{1}{2} |x|^{-\alpha} u^2\right)\right] = -\kappa \nabla \cdot \ima(\overline{v} \nabla v) - \ima (|x|^{-\alpha} u^2 \overline{v}), \\
		\partial_t \ima (\overline{v} \partial_k v) &= \frac{\kappa}{4} \partial_k \Delta(|v|^2) - \kappa \sum_{j=1}^d \partial_j \rea (\partial_j \overline{v} \partial_k v) \\
		&\quad+ 2\rea \left[\left(\gamma \overline{v} - \frac{1}{2} |x|^{-\alpha} \overline{u}^2 \right) \partial_k v\right] - \partial_k \rea \left[ \left(\gamma v-\frac{1}{2} |x|^{-\alpha} u^2\right) \overline{v}\right] \\
		&=\frac{\kappa}{4} \partial_k \Delta(|v|^2) - \kappa \sum_{j=1}^d \partial_j \rea (\partial_j \overline{v} \partial_k v) -\rea (|x|^{-\alpha} u^2 \partial_k \overline{v}) +\frac{1}{2} \partial_k \rea (|x|^{-\alpha} u^2 \overline{v}).
		\end{align*}
		Collecting the above identities and using the fact that
		\[
		\partial_k \rea(|x|^{-\alpha} u^2 \overline{v}) = \rea (\partial_k(|x|^{-\alpha}) u^2 \overline{v} + 2|x|^{-\alpha} u\partial_k u \overline{v} + |x|^{-\alpha} u^2 \partial_k \overline{v}),
		\]
		we obtain the desired identities.
\end{proof}

\subsection{A cutoff function}
Let $R>0$. We define the radial function
\begin{align} \label{chi-R}
\chi_R(x) =\chi_R(r) = R^2 \chi(r/R), \quad r=|x|.
\end{align}
with
\[
\chi(r) = \int_0^r \zeta(s) \dd s, 
\]
where $\zeta: [0,\infty) \rightarrow [0,\infty)$ satisfies
\[
\zeta(r) =\left\{
\renewcommand*{\arraystretch}{1.3}
\begin{array}{ccl}
2r &\text{if}& 0\leq r \leq  1, \\
2r - 2(r-1)^4 &\text{if}& 1 <r \leq 1+\frac{1}{\sqrt[3]{4}}, \\
\text{smooth and } \zeta'(r)<0 &\text{if}& 1+\frac{1}{\sqrt[3]{4}} < r \leq 2, \\
0 &\text{if}& r> 2.
\end{array}
\right.
\]
We collect some properties of $\chi_R$ in the following lemma.

\begin{lemma} \label{lem-prop-chi-R}
	We have
	\begin{align} \label{chi-R-prop-1}
	\|\nabla^j \chi_R\|_{L^\infty} &\lesssim R^{2-j}, \quad 0\leq j\leq 4, 
	\end{align}
	and
	\begin{align} \label{chi-R-prop-2}
	\supp(\nabla^j \chi_R) &\subset \left\{
	\renewcommand*{\arraystretch}{1.3}
	\begin{array}{ccl}
	\{|x| \leq 2R\} &\text{if}& j=1,2, \\
	\{R\leq |x| \leq 2R\} &\text{if}& j=3,4. 
	\end{array}
	\right.
	\end{align}
	and
	\begin{align} \label{chi-R-prop-3}
	\frac{\chi'_R(r)}{r} \leq 2, \quad \chi''_R(r) \leq 2, \quad \forall r\geq 0.
	\end{align}
	In addition, we have
	\begin{align} \label{chi-R-prop-4}
	\frac{\chi'_R(r)}{r} -\chi''_R(r) \geq 0, \quad \forall r \geq 0.
	\end{align}
\end{lemma}

\begin{proof}
	The estimates \eqref{chi-R-prop-1}--\eqref{chi-R-prop-3} follow directly from the choice of $\chi$. The estimate \eqref{chi-R-prop-4} follows from
	\begin{align*}
	\frac{\chi'_R(r)}{r} = \frac{\zeta(r/R)}{r/R} = \left\{
	\renewcommand*{\arraystretch}{1.3}
	\begin{array}{ccl}
	2 &\text{if}& 0\leq r/R \leq 1, \\
	2- 2\frac{(r/R-1)^4}{r/R} &\text{if}& 1< r/R \leq 1+\frac{1}{\sqrt[3]{4}}, \\
	\geq 0  &\text{if}& 1+\frac{1}{\sqrt[3]{4}} < r/R \leq 2, \\
	0&\text{if}& r/R > 2.
	\end{array}
	\right.
	\end{align*}
	and 
	\begin{align*}
	\chi''_R(r) = \zeta'(r/R) = \left\{
	\renewcommand*{\arraystretch}{1.3}
	\begin{array}{ccl}
	2 &\text{if} & 0\leq r/R \leq 1, \\
	2-8(r/R-1)^3 &\text{if}& 1 < r/R \leq 1+\frac{1}{\sqrt[3]{4}}, \\
	<0 &\text{if}& 1+\frac{1}{\sqrt[3]{4}} < r/R \leq 2, \\
	0 &\text{if}& r/R> 2.
	\end{array}
	\right.
	\end{align*}
\end{proof}

\subsection{Localized virial identity}
Let $\chi_R$ be as in \eqref{chi-R}. We define the localized virial quantity
\begin{align} \label{M-chi-R}
\Mcal_{\chi_R}(t) = \ima \int_{\Rb^d} \nabla \chi_{R} \cdot (\nabla u(t) \overline{u}(t) +\nabla v(t) \overline{v}(t)) \dd x.
\end{align}
Then we have (using Remark \ref{rem-loca-viri-rad})
\begin{align*}
\frac{\dd}{\dd t} \Mcal_{\chi_R}(t) &= \frac{1}{4} \int_{\Rb^d} \Delta^2\chi_R (|u(t)|^2 + \kappa |v(t)|^2) \dd x + \int_{\Rb^d} \frac{\chi'_R}{r} (|\nabla u(t)|^2 + \kappa |\nabla v(t)|^2) \dd x \\
&\quad + \int_{\Rb^d} \left(\frac{\chi''_R}{r^2}-\frac{\chi'_R}{r^3}\right) (|x\cdot \nabla u(t)|^2 + \kappa |x\cdot \nabla v(t)|^2) \dd x \\
&\quad -\frac{1}{2} \rea \int_{\Rb^d} \left( \chi''_R + (d-1+2\alpha)\frac{\chi'_R}{r}\right) |x|^{-\alpha} u^2(t) \overline{v}(t) \dd x.
\end{align*}
We can rewrite it as
\begin{align*}
\frac{\dd}{\dd t} \Mcal_{\chi_R}(t) &= 2 \Gb(u(t),v(t)) + \frac{1}{4} \int_{\Rb^d} \Delta^2\chi_R (|u(t)|^2 + \kappa |v(t)|^2) \dd x \\
&\quad - \int_{\Rb^d} \left(2-\frac{\chi'_R}{r}\right) (|\nabla u(t)|^2 + \kappa |\nabla v(t)|^2) \dd x \\
&\quad + \int_{\Rb^d} \left(\frac{\chi''_R}{r^2}-\frac{\chi'_R}{r^3}\right) (|x\cdot \nabla u(t)|^2 + \kappa |x\cdot \nabla v(t)|^2) \dd x \\
&\quad + \frac{1}{2} \rea \int_{\Rb^d} \left( (2-\chi''_R) + (d-1+2\alpha) \left(2-\frac{\chi'_R}{r}\right)\right) |x|^{-\alpha} u^2(t) \overline{v}(t) \dd x \\
&= 2\Gb(u(t),v(t)) + E_1(u(t),v(t)) + E_2(u(t),v(t)) +E_3(u(t),v(t)),
\end{align*}
where
\begin{align} \label{E123}
\begin{aligned}
E_1(u,v) &= \frac{1}{4} \int_{\Rb^d} \Delta^2\chi_R (|u|^2 + \kappa |v|^2) \dd x, \\
E_2(u,v) &= - \int_{\Rb^d} \left(2-\frac{\chi'_R}{r}\right) (|\nabla u|^2 + \kappa |\nabla v|^2) \dd x + \int_{\Rb^d} \left(\frac{\chi''_R}{r^2}-\frac{\chi'_R}{r^3}\right) (|x\cdot \nabla u|^2 + \kappa |x\cdot \nabla v|^2) \dd x, \\
E_3(u,v) &= \frac{1}{2} \rea \int_{\Rb^d} \left( (2-\chi''_R) + (d-1+2\alpha) \left(2-\frac{\chi'_R}{r}\right)\right) |x|^{-\alpha} u^2 \overline{v} \dd x.
\end{aligned}
\end{align}

\subsection{Mass-critical blow-up solutions}
Before giving the proof of Theorem \ref{theo-blow-crit}, we need the following preliminary lemma.
\begin{lemma}
	Define
	\begin{align} \label{chi-12R}
	\chi_{1R}(r):= 2-\frac{\chi'_R(r)}{r}, \quad \chi_{2R}(r):= 2-\chi''_R(r) + 3 \left(2-\frac{\chi'_R(r)}{r}\right).
	\end{align}
	The following properties hold:
	\begin{align}
	|\chi_{2R}(r)| &\lesssim 1,  \quad \forall r\geq 0 \label{chi-12R-prop-1} \\
	|\nabla(\chi_{2R}^{\frac{2}{3}}(r))| &\lesssim R^{-1}, \quad \forall r\geq 0, \label{chi-12R-prop-2}
	\end{align}
	and for $R>0$ sufficiently large,
	\begin{align} \label{chi-12R-prop-3}
	\chi_{1R}(r) - CR^{-\alpha} \chi_{2R}^{\frac{4}{3}}(r) \geq 0, \quad \forall r\geq 0.
	\end{align}
\end{lemma}

\begin{proof}
	We have
	\begin{align*}
	2-\frac{\chi'_R(r)}{r} = \left\{
	\renewcommand*{\arraystretch}{1.3}
	\begin{array}{ccl}
	0 &\text{if}& 0\leq r/R \leq 1, \\
	2\frac{(r/R-1)^4}{r/R} &\text{if}& 1< r/R \leq 1+\frac{1}{\sqrt[3]{4}}, \\
	2-\frac{\zeta(r/R)}{r/R}  &\text{if}& 1+\frac{1}{\sqrt[3]{4}} < r/R \leq 2, \\
	2&\text{if}& r/R > 2,
	\end{array}
	\right.
	\end{align*}
	and 
	\begin{align*}
	2-\chi''_R(r) = \left\{
	\renewcommand*{\arraystretch}{1.3}
	\begin{array}{ccl}
	0 &\text{if} & 0\leq r/R \leq 1, \\
	8(r/R-1)^3 &\text{if}& 1 < r/R \leq 1+\frac{1}{\sqrt[3]{4}}, \\
	2-\zeta'(r/R) &\text{if}& 1+\frac{1}{\sqrt[3]{4}} < r/R \leq 2, \\
	2 &\text{if}& r/R > 2.
	\end{array}
	\right.
	\end{align*}
	Note that $\zeta'(r/R) <0$ for all $1+\frac{1}{\sqrt[3]{4}} < r/R \leq 2$. From this, we have \eqref{chi-12R-prop-1}. 
	
	To verify \eqref{chi-12R-prop-2}, we consider several cases. 
	If $0\leq r/R \leq 1$ or $r/R>2$, then $\chi_{2R}(r)$ is constant. So \eqref{chi-12R-prop-2} holds trivially. 
	
	If $1<r/R\leq 1+\frac{1}{\sqrt[3]{4}}$, we have
	\[
	\chi_{2R}(r) = 8(r/R-1)^3 + 6 \frac{(r/R-1)^4}{r/R} = (r/R-1)^3 \left(14 - \frac{6}{r/R}\right).
	\]
	Thus
	\[
	\chi_{2R}^{\frac{2}{3}}(r) = (r/R-1)^2 \left(14-\frac{6}{r/R}\right)^{\frac{2}{3}} =: h(r/R),
	\]
	where
	\[
	h(\lambda)= (\lambda-1)^2 \left(14-\frac{6}{\lambda}\right)^{\frac{2}{3}}.
	\]
	It follows that
	\[
	|\nabla (\chi_{2R}^{\frac{2}{3}}(r))| = |\partial_r(\chi_{2R}^{\frac{2}{3}}(r))| = \frac{1}{r} |h'(r/R)|,
	\]
	where
	\[
	h'(\lambda) = 2(\lambda-1) \left(14-\frac{6}{\lambda}\right)^{-\frac{1}{3}} \left(14-\frac{4}{\lambda} - \frac{2}{\lambda^2}\right).
	\]
	We readily see that $|h'(\lambda)| \leq C$ for all $1<\lambda \leq 1+\frac{1}{\sqrt[3]{4}}$. Thus \eqref{chi-12R-prop-2} holds in this range. 
	
	If $1+\frac{1}{\sqrt[3]{4}} <r/R \leq 2$, then $\chi''_{R}(r) =\zeta'(r/R)<0$ and
	\[
	\frac{\chi'_R(r)}{r} = \frac{\zeta(r/R)}{r/R} \in \left[\frac{\zeta(2)}{2}, \frac{\zeta(1+1/\sqrt[3]{4})}{1+1/\sqrt[3]{4}}\right) = \left[0, \frac{3+4\sqrt[3]{4}}{2+2\sqrt[3]{4}} \right).
	\]
	Hence
	\[
	\chi_{2R}(r) \geq 2+\frac{3}{2+2\sqrt[3]{4}}, \quad \forall~ 1+\frac{1}{\sqrt[3]{4}} < r/R \leq 2.
	\]  
	As one can readily check that $|\nabla \chi_{2R}(r)| \lesssim R^{-1}$ and $\nabla (\chi_{2R}^{\frac{2}{3}}(r)) = \frac{2}{3} \frac{\nabla \chi_{2R}(r)}{\chi_{2R}^{1/3}(r)}$, we deduce that \eqref{chi-12R-prop-2} holds for $1+\frac{1}{\sqrt[3]{4}}<r/R \leq 2$. This finishes the proof of \eqref{chi-12R-prop-2}.
	
	Finally, let us show \eqref{chi-12R-prop-3}. As above, we consider three cases. If $0\leq r/R \leq 1$, then \eqref{chi-12R-prop-3} is obvious as $\chi_{1R}(r)=\chi_{2R}(r)=0$.
	
	If $1<r/R \leq 1+\frac{1}{\sqrt[3]{4}}$, we have
	\[
	\chi_{1R}(r) = 8\frac{(r/R-1)^4}{r/R}
	\]
	and
	\[
	\chi_{2R}(r) = (r/R-1)^3 \left(14-\frac{6}{r/R}\right) < (r/R-1)^3 \left(14-\frac{6\sqrt[3]{4}}{1+\sqrt[3]{4}}\right).
	\]
	Thus we get
	\begin{align*}
	\chi_{1R}(r)-CR^{-\alpha} \chi_{2R}^{\frac{4}{3}}(r) &> 8\frac{(r/R-1)^4}{r/R} - CR^{-\alpha} (r/R-1)^4 \left(14-\frac{6\sqrt[3]{4}}{1+\sqrt[3]{4}}\right)^{\frac{4}{3}} \\
	&=(r/R-1)^4 \left(\frac{8}{r/R}- CR^{-\alpha}\left(14-\frac{6\sqrt[3]{4}}{1+\sqrt[3]{4}}\right)^{\frac{4}{3}}\right).
	\end{align*}
	As $1<r/R\leq 1+\frac{1}{\sqrt[3]{4}}$, by taking $R>0$ sufficiently large, we get
	\[
	\chi_{1R}(r)-CR^{-\alpha}\chi_{2R}^{\frac{4}{3}}(r) \geq (r/R-1)^4.
	\]
	
	If $r/R >1+\frac{1}{\sqrt[3]{4}}$, we have $\zeta'(r/R) \leq 0$, hence
	\[
	\chi_{1R}(r) =2-\frac{\chi'_R(r)}{r} =2-\frac{\zeta(r/R)}{r/R} \geq 2 -\frac{\zeta(1+1/\sqrt[3]{4})}{1+1/\sqrt[3]{4}}=\frac{1}{2+2\sqrt[3]{4}}.
	\] 
	On the other hand, we have $|\chi_{2R}(r)| \lesssim 1$ (by \eqref{chi-12R-prop-1}). Therefore, \eqref{chi-12R-prop-3} holds provided $R>0$ is taken sufficiently large.  The proof is complete.
\end{proof}

\begin{proof}[Proof of Theorem \ref{theo-blow-crit}]
	Let $\Mcal_{\chi_R}(t)$ be as in \eqref{M-chi-R}. We have
	\[
	\frac{\dd}{\dd t} \Mcal_{\chi_R}(t) = 2 \Gb(u(t),v(t)) + E_1(u(t),v(t)) + E_2(u(t),v(t)) + E_3(u(t),v(t)), \quad \forall t\in (-T_*,T^*),
	\]
	where $E_1, E_2, E_3$ are as in \eqref{E123}. 
	By the conservation of mass, we infer from \eqref{chi-R-prop-1} that
	\[
	E_1(u(t),v(t)) \leq C(\kappa, \Mb(u_0,v_0)) R^{-2}.
	\]
	Using \eqref{chi-R-prop-4}, we have
	\[
	E_2(u(t),v(t)) \leq - \int_{\Rb^d} \chi_{1R} (|\nabla u(t)|^2 + \kappa |\nabla v(t)|^2) \dd x, 
	\]
	where $\chi_{1R}$ is as in \eqref{chi-12R}. From \eqref{chi-12R}, we also have
	\[
	E_3(u(t),v(t)) = \frac{1}{2} \rea \int_{\Rb^d} \chi_{2R} |x|^{-\alpha} u^2(t) \overline{v}(t) \dd x.
	\]
	As $\supp(\chi_{2R})\subset \{|x| \geq R\}$, we see that
	\begin{align*}
	E_3(u(t),v(t)) &\lesssim R^{-\alpha} \left(\int_{\Rb^d} \chi_{2R} |u(t)|^3 \dd x\right)^{\frac{2}{3}} \left(\int_{\Rb^d} \chi_{2R} |v(t)|^3 \dd x \right)^{\frac{1}{3}} \\
	&\lesssim R^{-\alpha} \left( \int_{\Rb^d} \chi_{2R} |u(t)|^3 \dd x + \int_{\Rb^d} \chi_{2R} |v(t)|^3 \dd x\right),
	\end{align*}
	where we have used the inequality $a^2b\leq \frac{1}{3} (2a^3+b^3)$ for all $a,b\geq 0$. 
	
	Let us first consider the case $d=3$ which corresponds to $\alpha=\frac{1}{2}$. We estimate
	\begin{align*}
	\int_{\Rb^d} \chi_{2R} |u(t)|^3 \dd x &\leq \|\chi_{2R} |u(t)|^{\frac{3}{2}}\|_{L^4} \||u(t)|^{\frac{3}{2}}\|_{L^{\frac{4}{3}}} \\
	&\leq \|\chi_{2R}^{\frac{2}{3}} u(t)\|^{\frac{3}{2}}_{L^6} \|u(t)\|^{\frac{3}{2}}_{L^2} \\
	&\leq C(\Mb(u_0,v_0)) \|\nabla (\chi_{2R}^{\frac{2}{3}} u(t))\|^{\frac{3}{2}}_{L^2} \\
	&\leq C(\Mb(u_0,v_0)) \left(\|\nabla (\chi_{2R}^{\frac{2}{3}} u(t))\|^2_{L^2}+1\right) \\
	&\leq C(\Mb(u_0,v_0)) \left( \|\chi_{2R}^{\frac{2}{3}} \nabla u(t)\|^2_{L^2} + \|\nabla(\chi_{2R}^{\frac{2}{3}}) u(t)\|^2_{L^2}+1\right).
	\end{align*}
	Similarly, we have
	\[
	\int_{\Rb^d} \chi_{2R} |v(t)|^3 \dd x \leq C(\kappa,\Mb(u_0,v_0)) \left(\kappa \|\chi_{2R}^{\frac{2}{3}} \nabla v(t)\|^2_{L^2} + \kappa \|\nabla (\chi_{2R}^{\frac{2}{3}}) v(t)\|^2_{L^2} +1 \right).
	\]
	Thus we get
	\begin{align*}
	E_3(u(t),v(t)) &\leq CR^{-\alpha} \left(\|\chi_{2R}^{\frac{2}{3}} \nabla u(t)\|^2_{L^2} + \kappa \|\chi_{2R}^{\frac{2}{3}} \nabla v(t)\|^2_{L^2}\right) \\
	&\quad + C R^{-\alpha} \left(\|\nabla(\chi_{2R}^{\frac{2}{3}}) u(t)\|^2_{L^2} + \kappa \|\nabla(\chi_{2R}^{\frac{2}{3}} v(t)\|^2_{L^2}\right) + CR^{-\alpha}
	\end{align*}
	for some constant $C=C(\kappa,\Mb(u_0,v_0))>0$. It follows that
	\begin{align*}
	\frac{\dd}{\dd t} \Mcal_{\chi_R}(t) &\leq 2\Gb(u(t),v(t)) - \int_{\Rb^d} \left(\chi_{1R} - CR^{-\alpha} \chi_{2R}^{\frac{4}{3}}\right) (|\nabla u(t)|^2 + \kappa |\nabla v(t)|^2) \dd x \\
	&\quad + CR^{-\alpha} \int_{\Rb^d} |\nabla(\chi_{2R}^{\frac{2}{3}})|^2 (|u(t)|^2+\kappa |v(t)|^2) \dd x + CR^{-\alpha}, \quad \forall t\in (-T_*,T^*).
	\end{align*}
	Thanks to \eqref{blow-cond}, \eqref{chi-12R-prop-2}, \eqref{chi-12R-prop-3}, and the conservation of mass, we take $R>0$ sufficiently large to get
	\[
	\frac{\dd}{\dd t} \Mcal_{\chi_R}(t) \leq -\delta, \quad \forall t\in (-T_*,T^*).
	\]
	
	$\bullet$ In the mass-resonance case, i.e., $\kappa=\frac{1}{2}$, we have
	\[
	\frac{\dd^2}{\dd t^2}\Vc_{\chi_R}(t) =\frac{\dd}{\dd t}\Mcal_{\chi_R}(t) \leq -\delta, \quad \forall t\in(-T_*,T^*),
	\]
	where 
	\begin{align} \label{V-chi-R}
	\Vc_{\chi_R}(t) = \int_{\Rb^d} \chi_R (|u(t)|^2 + 2|v(t)|^2) \dd x.
	\end{align}
	Taking the integration over $[0,t]$, we obtain
	\[
	\Vc_{\chi_R}(t) \leq \Vc_{\chi_R}(0) + t \frac{\dd}{\dd t} \Vc_{\chi_R}(0) -\frac{\delta}{2} t^2, \quad \forall t\in [0,T^*).
	\]
	Assume by contradiction that $T^*=\infty$. As $\delta>0$, there exists $t_*>0$ sufficiently large such that $\Vc_{\chi_R}(t_*) <0$ which is not possible, hence $T^*<\infty$. 
	
	$\bullet$ In the non mass-resonance case, i.e., $\kappa \ne \frac{1}{2}$, if $T^*<\infty$, we are done. Otherwise, if $T^*=\infty$, then we have
	\[
	\frac{\dd}{\dd t} \Mcal_{\chi_R}(t) \leq -\delta, \quad \forall t \in [0,\infty).
	\]
	Integrating over $[0,t]$, we get
	\[
	\Mcal_{\chi_R}(t) \leq \Mcal_{\chi_R}(0) -\delta t\leq -\frac{\delta}{2} t, \quad \forall t\geq t_0:=\frac{|\Mcal_{\chi_R}(0)|}{\delta}.
	\]
	By H\"older's  inequality, we infer that
	\begin{align*}
	\frac{\delta}{2} t \leq |\Mcal_{\chi_R}(t)| &\leq \|\nabla \chi_R\|_{L^\infty} (\|\nabla u(t)\|_{L^2} \|u(t)\|_{L^2} + \|\nabla v(t)\|_{L^2} \|v(t)\|_{L^2}) \\
	&\leq C(R, \kappa, \Mb(u_0,v_0)) \sqrt{\Kb(u(t),v(t))}, \quad \forall t\geq t_0
	\end{align*}
	which yields $\Kb(u(t),v(t)) \geq Ct^2$ for some constant $C=C(R, \kappa, \delta, \Mb(u_0,v_0))>0$. 
	
	This completes the proof for positive times in three dimensions. The one for negative times is treated in a similar manner.
	
	The proof is similar when $d=2$ and $\alpha=1$. We need to use the Sobolev embedding $H^1(\Rb^2) \subset L^6(\Rb^2)$ to get
	\begin{align*}
	\int_{\Rb^d} \chi_{2R} |u(t)|^3 \dd x &\leq \|\chi_{2R} |u(t)|^{\frac{3}{2}}\|_{L^4} \||u(t)|^{\frac{3}{2}}\|_{L^{\frac{4}{3}}} \\
	&\leq \|\chi_{2R}^{\frac{2}{3}} u(t)\|^{\frac{3}{2}}_{L^6} \|u(t)\|^{\frac{3}{2}}_{L^2} \\
	&\leq C(\Mb(u_0,v_0)) \|\chi_{2R}^{\frac{2}{3}} u(t)\|^{\frac{3}{2}}_{H^1} \\
	&\leq C(\Mb(u_0,v_0)) \left(\|\chi_{2R}^{\frac{2}{3}} u(t)\|^2_{H^1} +1\right) \\
	&\leq C(\Mb(u_0,v_0)) \left(\|\nabla (\chi_{2R}^{\frac{2}{3}} u(t))\|^2_{L^2}+\|\chi_{2R}^{\frac{2}{3}} u(t)\|^2_{L^2}+1\right) \\
	&\leq C(\Mb(u_0,v_0)) \left( \|\chi_{2R}^{\frac{2}{3}} \nabla u(t)\|^2_{L^2} + \|\nabla(\chi_{2R}^{\frac{2}{3}}) u(t)\|^2_{L^2}+\|\chi_{2R}^{\frac{2}{3}} u(t)\|^2_{L^2}+1\right).
	\end{align*}
	Similarly, we also have
	\[
	\int_{\Rb^d} \chi_{2R} |v(t)|^3 \dd x \leq C(\kappa,\Mb(u_0,v_0)) \left(\kappa \|\chi_{2R}^{\frac{2}{3}} \nabla v(t)\|^2_{L^2} +\kappa \|\nabla(\chi_{2R}^{\frac{2}{3}}) v(t)\|^2_{L^2}+ \kappa\|\chi_{2R}^{\frac{2}{3}} v(t)\|^2_{L^2}+1\right).
	\]
	The remaining argument is exactly the same as above as the additional term can be estimated, using \eqref{chi-12R-prop-1} and the mass conservation, as
	\[
	\|\chi_{2R}^{\frac{2}{3}} u(t)\|^2_{L^2} + \kappa \|\chi_{2R}^{\frac{2}{3}} v(t)\|^2_{L^2} \leq C(\kappa,\Mb(u_0,v_0)).
	\]
	
	Finally, we will verify that \eqref{blow-cond} is satisfied if $\Hb(u_0,v_0)<0$. It follows directly from the conservation of mass and energy, and the fact (see \eqref{est-Hb}) that
	\[
	\Gb(u(t),v(t)) = \Kb(u(t),v(t))-2\Pb(u(t),v(t)) = 2\left(\Eb(u(t),v(t)) - \gamma \|v(t)\|^2_{L^2}\right) \leq 2 \Hb(u(t),v(t)). 
	\]
	The proof is complete.
\end{proof}

\subsection{Mass-supercritical blow-up solutions}
Before proving Theorem \ref{theo-blow-supe}, we have the following observation.
\begin{lemma} \label{lem-obse-Gb}
	Let $2\leq d\leq 5$, $0<\alpha<\min\{2,d\}$, $\alpha>\frac{4-d}{2}$, $\alpha<\frac{6-d}{2}$ if $3\leq d\leq 5$, $\kappa>0$, and $\gamma \in \Rb$. Let $(u_0,v_0) \in \Hc^1$ and $(u,v) \in C((-T_*,T^*), \Hc^1)$ be the corresponding maximal solution to \eqref{INLS}. Assume that \eqref{blow-cond} holds. Then there exists $\vareps_0 =\vareps_0(\delta)>0$ such that
	\begin{align} \label{est-solu-blow-supe}
	\Gb(u(t),v(t)) + \vareps \Kb(u(t),v(t)) \leq -\frac{\delta}{2}, \quad \forall t\in (-T_*,T^*)
	\end{align}
	for all $0<\vareps \leq \vareps_0$. In addition, 
	\begin{align} \label{inf-solu-supe}
	\inf_{t\in (-T_*,T^*)} \Kb(u(t),v(t)) \geq C
	\end{align}
	for some constant $C>0$.
\end{lemma}
\begin{proof}
	By \eqref{est-Hb}, and the conservation laws of mass and energy, we have for all $t\in (-T_*,T^*)$,
	\begin{align}
	\Gb(u(t),v(t)) &=\Kb(u(t),v(t)) - \frac{d+2\alpha}{2} \Pb(u(t),v(t)) \nonumber\\
	&= \frac{d+2\alpha}{2} \left(\Eb(u(t),v(t)) -  \gamma \|v(t)\|^2_{L^2}\right) -\frac{d+2\alpha-4}{4} \Kb(u(t),v(t)) \nonumber\\
	&\leq \frac{d+2\alpha}{2} \Hb(u(t),v(t)) -\frac{d+2\alpha-4}{4} \Kb(u(t),v(t)). \label{est-GHK}
	\end{align}
	For $\vareps>0$ small, we infer from \eqref{est-GHK} that
	\begin{align*}
	\Gb(u(t),v(t)) + \vareps \Kb(u(t),v(t)) &\leq \Gb(u(t),v(t)) + \vareps \left( \frac{2(d+2\alpha)}{d+2\alpha-4} \Hb(u(t),v(t)) - \frac{4}{d+2\alpha-4} \Gb(u(t),v(t))\right) \\
	&= \left(1-\frac{4\vareps}{d+2\alpha-4}\right)\Gb(u(t),v(t)) + \frac{2\vareps(d+2\alpha)}{d+2\alpha-4} \Hb(u(t),v(t))
	\end{align*}
	which together with \eqref{blow-cond}, and the conservation laws of mass and energy imply
	\[
	\Gb(u(t),v(t)) +\vareps \Kb(u(t),v(t)) \leq -\left(1-\frac{4\vareps}{d+2\alpha-4}\right)\delta  + \frac{2\vareps(d+2\alpha)}{d+2\alpha-4} \Hb(u_0,v_0) \leq -\frac{\delta}{2} 
	\]
	provided that $0<\vareps \leq \vareps_0$ with some $\vareps_0=\vareps_0(\delta)>0$.	This proves \eqref{est-solu-blow-supe}.	
	
	To see \eqref{est-solu-blow-supe}, we argue by contradiction. Assume that there exists a time sequence $(t_n)_n\subset (-T_*,T^*)$ such that  $\Kb(u(t_n),v(t_n)) \rightarrow 0$. By \eqref{GN-ineq} and the conservation of mass, we readily see that $\Gb(u(t_n),v(t_n))\rightarrow 0$. This contradicts \eqref{blow-cond}. 
\end{proof}

\begin{proof}[Proof of Theorem \ref{theo-blow-supe}]
	We will consider separately two cases: $2\leq d\leq 4$ and $d=5$.
	
	{\bf Case 1. $2\leq d\leq 4$.} Let $\Mcal_{\chi_R}(t)$ be as in \eqref{M-chi-R}. We have
	\begin{align*}
	\frac{\dd}{\dd t} \Mcal_{\chi_R}(t) = 2\Gb(u(t),v(t)) + E_1(u(t),v(t)) + E_2(u(t),v(t)) + E_3(u(t),v(t)), \quad \forall t\in (-T_*,T^*),
	\end{align*}
	with $E_1, E_2, E_3$ as in \eqref{E123}. By \eqref{chi-R-prop-1}, the conservation of mass implies
	\[
	E_1(u(t),v(t)) \leq C(\kappa, \Mb(u_0,v_0)) R^{-2}.
	\]
	Thanks to \eqref{chi-R-prop-3} and \eqref{chi-R-prop-4}, we have $E_2(u(t),v(t)) \leq 0$. On the other hand, from \eqref{chi-R-prop-2} and \eqref{chi-R-prop-3}, we have
	\begin{align*}
	E_3(u(t),v(t)) &\leq C \int_{|x|\geq R} |x|^{-\alpha} |u(t)|^2 |v(t)| \dd x \\
	&\leq C R^{-\alpha} \|u(t)\|^2_{L^3} \|v(t)\|_{L^3}.
	\end{align*}
	Using the standard Gagliardo--Nirenberg inequality
	\[
	\|f\|^3_{L^3} \leq C\|\nabla f\|^{\frac{d}{2}}_{L^2} \|f\|^{\frac{6-d}{2}}_{L^2},
	\]
	we get
	\begin{align*}
	E_3(u(t),v(t)) &\leq CR^{-\alpha} \|\nabla u(t)\|^{\frac{d}{3}}_{L^2} \|u(t)\|^{\frac{6-d}{3}}_{L^2} \|\nabla v(t)\|^{\frac{d}{6}}_{L^2} \|v(t)\|^{\frac{6-d}{6}}_{L^2} \\
	&\leq C(\kappa, \Mb(u_0,v_0)) R^{-\alpha} (\Kb(u(t),v(t))^{\frac{d}{4}} \\
	&\leq \left\{
	\renewcommand*{\arraystretch}{1.3}
	\begin{array}{ccl}
	C(\kappa, \Mb(u_0,v_0)) R^{-\alpha} (\Kb(u(t),v(t)) +1) &\text{if}& d=2,3, \\
	C(\kappa, \Mb(u_0,v_0)) R^{-\alpha} \Kb(u(t),v(t)) &\text{if}& d=4.
	\end{array}
	\right.
	\end{align*}
	Collecting the above estimates, we obtain for all $t\in (-T_*,T^*)$,
	\[
	\frac{\dd}{\dd t} \Mcal_{\chi_R}(t) \leq 2 \Gb(u(t),v(t)) + \left\{
	\renewcommand*{\arraystretch}{1.3}
	\begin{array}{ccl}
	C R^{-\alpha} \Kb(u(t),v(t)) + C R^{-\alpha} &\text{if}& d=2,3, \\
	C R^{-\alpha} \Kb(u(t),v(t)) + C R^{-2} &\text{if}& d=4,
	\end{array}
	\right.
	\]
	for some constant $C$ depending on $\kappa$ and $\Mb(u_0,v_0)$.
	
	Using \eqref{est-solu-blow-supe}, there exists $\vareps_0>0$ such that for all $t\in (-T_*,T^*)$,
	\[
	\frac{\dd}{\dd t} \Mcal_{\chi_R}(t) \leq -(2\vareps_0 -CR^{-\alpha}) \Kb(u(t),v(t)) - \delta + \left\{
	\renewcommand*{\arraystretch}{1.3}
	\begin{array}{ccl}
	C R^{-\alpha} &\text{if}& d=2,3, \\
	C R^{-2} &\text{if}& d=4.
	\end{array}
	\right.
	\]
	Taking $R>0$ sufficiently large, we get 
	\[
	\frac{\dd}{\dd t} \Mcal_{\chi_R}(t) \leq -\vareps_0 \Kb(u(t),v(t)) - \frac{\delta}{2}, \quad \forall t\in (-T_*,T^*).
	\]
	From this, we deduce the blow-up as follows. 
	
	$\bullet$ In the mass-resonance case, i.e., $\kappa=\frac{1}{2}$, we have
	\[
	\frac{\dd^2}{\dd t^2}\Vc_{\chi_R}(t) =\frac{\dd}{\dd t} \Mcal_{\chi_R}(t) \leq -\frac{\delta}{2},\quad \forall t\in (-T_*,T^*),
	\]
	where $\Vc_{\chi_R}$ is as in \eqref{V-chi-R}. The same argument as in the proof of Theorem \ref{theo-blow-crit} yields $T_*,T^*<\infty$. 
	
	$\bullet$ In the non mass-resonance, i.e., $\kappa \ne \frac{1}{2}$, we argue by contradiction. Assume that $T^*=\infty$. We have
	\begin{align} \label{est-M-chi-R}
	\frac{\dd}{\dd t} \Mcal_{\chi_R}(t) \leq -\vareps_0 \Kb(u(t),v(t)) - \frac{\delta}{2}, \quad \forall t\in [0,\infty).
	\end{align}
	Integrating this inequality over $[0,t]$, we obtain
	\[
	\Mcal_{\chi_R}(t) \leq \Mcal_{\chi_R}(0) - \frac{\delta}{2} t \leq 0, \quad \forall t\geq t_0:= \frac{2|\Mcal_{\chi_R}(0)|}{\delta}.
	\]
	Integrating \eqref{est-M-chi-R} over $[t_0,t]$, we get
	\[
	\Mcal_{\chi_R}(t) \leq -\vareps_0 \int_{t_0}^t \Kb(u(s),v(s))ds, \quad \forall t\geq t_0.
	\]
	On the other hand, we have
	\begin{align*}
	|\Mcal_{\chi_R}(t)| &\leq \|\nabla \chi_R\|_{L^\infty} \left(\|\nabla u(t)\|_{L^2} \|u(t)\|_{L^2} + \|\nabla v(t)\|_{L^2} \|v(t)\|_{L^2}\right) \\
	&\leq C(R, \kappa, \Mb(u_0,v_0)) \sqrt{\Kb(u(t),v(t))}
	\end{align*}
	which implies
	\[
	\Mcal_{\chi_R}(t) \leq -A \int_{t_0}^t |\Mcal_{\chi_R}(s)|^2 ds, \quad \forall t\geq t_0,
	\]
	where $A=A(\vareps_0, R, \kappa, \delta, \Mb(u_0,v_0))>0$. Set 
	\[
	y(t) = \int_{t_0}^t |\Mcal_{\chi_R}(s)|^2 ds. 
	\]
	We see that $y$ is strictly increasing and non-negative and
	\[
	y'(t) = |\Mcal_{\chi_R}(t)|^2 \geq A^2 y^2(t).
	\]
	Fix some $t_1>t_0$. We integrate this inequality over $[t_1,t]$ and get
	\[
	y(t) \geq \frac{y(t_1)}{1-A^2 y(t_1)(t-t_1)}, \quad \forall t\geq t_1.
	\]
	Thus we have
	\[
	y(t) \rightarrow +\infty \text{ as } t\nearrow t_* :=t_1 +\frac{1}{A^2 y(t_1)}>t_1.
	\]
	Hence $\Mcal_{\chi_R}(t) \leq -A y(t) \rightarrow -\infty$ as $t\nearrow t_*$. Therefore, the solution cannot exist for all time $t\geq 0$ and, consequently, we must have $T^*<\infty$. A similar argument goes for negative times. This finishes the proof for $2\leq d\leq 4$.
	
	{\bf Case 2. $d=5$.} We only consider the positive times since the one for negative times is treated in a similar manner. If $T^*<\infty$, we are done. Otherwise, if $T^*=\infty$, we first show \eqref{blow-infi-supe}. Assume by contradiction that 
	\begin{align} \label{boun-Kb}
	\sup_{t\in [0,\infty)} \Kb(u(t),v(t)) <\infty.
	\end{align}
	Arguing as above, we have
	\[
	\frac{\dd}{\dd t} \Mcal_{\chi_R}(t) \leq 2\Gb(u(t),v(t)) + C R^{-\alpha} \left(\Kb(u(t),v(t))\right)^{\frac{5}{4}} + C R^{-2}, \quad \forall t\in [0,\infty).
	\]
	By \eqref{est-solu-blow-supe}, there exists $\vareps_0>0$ such that 
	\begin{align} \label{blow-infi-supe-prof}
	\frac{\dd}{\dd t}\Mcal_{\chi_R}(t) \leq -2\vareps_0 \Kb(u(t),v(t)) -\delta + CR^{-\alpha} \left(\Kb(u(t),v(t))\right)^{\frac{5}{4}} + CR^{-2}, \quad \forall t\in [0,\infty).
	\end{align}
	By \eqref{boun-Kb}, we take $R>0$ sufficiently large and deduce
	\[
	\frac{\dd}{\dd t}\Mcal_{\chi_R}(t) \leq -\vareps_0\Kb(u(t),v(t))-\frac{\delta}{2}, \quad \forall t\in [0,\infty).
	\]
	From this, we can argue exactly as in Case 1 to show that there exists a finite time $t_*>0$ such that $\Mcal_{\chi_R}(t) \rightarrow -\infty$ as $t\nearrow t_*$. This is a contradiction. Thus we have proved \eqref{blow-infi-supe}.
	
	Let us now prove \eqref{blow-rate-T}. To this end, we fix $T>0$ and set
	\[
	R(T) := B \sup_{t\in[0,T]} \left(\Kb(u(t),v(t))\right)^{\frac{1}{4}}
	\]
	for some $B>0$ to be chosen later. Applying \eqref{blow-infi-supe} with $R=R(T)$, we have
	\[
	\frac{\dd}{\dd t} \Mcal_{\chi_{R(T)}}(t) \leq -2\vareps_0 \Kb(u(t),v(t)) + C(R(T))^{-\alpha} \left( \Kb(u(t),v(t))\right)^{\frac{5}{4}} + C(R(T))^{-2}, \quad \forall t\in [0,\infty).
	\]
	We observe that
	\begin{align*}
	C(R(T))^{-\alpha} \left(\Kb(u(t),v(t))\right)^{\frac{1}{4}} &= \frac{CB^{-\alpha} \left(\Kb(u(t),v(t))\right)^{\frac{1}{4}}}{\sup_{t\in[0,T]} \left(\Kb(u(t),v(t))\right)^{\frac{1}{4}}} \\
	&\leq CB^{-\alpha}, \quad \forall t\in [0,T] 
	\end{align*}
	and
	\begin{align*}
	\frac{C(R(T))^{-2}}{\Kb(u(t),v(t))} &= \frac{CB^{-2}}{\Kb(u(t),v(t))\sup_{t\in[0,T]} \left(\Kb(u(t),v(t))\right)^{\frac{1}{2}}} \\
	&\leq \frac{CB^{-2}}{C_0^{3/2}}, \quad \forall t\in [0,T],
	\end{align*}
	where $C_0$ is as in \eqref{inf-solu-supe}. Thus we get
	\[
	\frac{\dd}{\dd t} \Mcal_{\chi_{R(T)}}(t) \leq -\left(2\vareps_0 - CB^{-\alpha} - \frac{CB^{-2}}{C_0^{3/4}}\right) \Kb(u(t),v(t)), \quad \forall t\in [0,T].
	\]
	Taking $B>0$ sufficiently large, we obtain
	\begin{align} \label{est-RT}
	\frac{\dd}{\dd t} \Mcal_{\chi_{R(T)}}(t) \leq -\vareps_0 \Kb(u(t),v(t)), \quad \forall t\in [0,T].
	\end{align}
	As above, we consider two cases: mass-resonance and non mass-resonance.
	
	$\bullet$ In the mass-resonance case, we have
	\[
	\frac{\dd^2}{\dd t^2} \Vc_{\chi_{R(T)}}(t) =\frac{\dd}{\dd t} \Mcal_{\chi_{R(T)}}(t) \leq -\vareps_0 \Kb(u(t),v(t)), \quad \forall t\in [0,T].
	\]
	It follows that
	\[
	0\leq \Vc_{\chi_{R(T)}}(T) \leq \Vc_{\chi_{R(T)}}(0) + T \frac{\dd}{\dd t} \Vc_{\chi_{R(T)}}(0) -\vareps_0 \int_0^T \int_0^s \Kb(u(\tau),v(\tau)) \dd\tau\dd s.
	\]
	Observe that
	\[
	\Vc_{\chi_{R(T)}}(0) = \int_{\Rb^d} \chi_{R(T)} (|u_0|^2+ 2|v_0|^2) \dd x\leq C(u_0,v_0) (R(T))^2
	\]
	and
	\[
	\frac{\dd}{\dd t}\Vc_{\chi_{R(T)}}(0) = \ima \int_{\Rb^d} \nabla \chi_{R(T)} \cdot (\nabla u_0 \overline{u}_0 + 2\kappa \nabla v_0 \overline{v}_0) \dd x \leq C(\kappa, u_0,v_0) R(T).
	\]
	Thanks to \eqref{inf-solu-supe}, we deduce
	\[
	\vareps_0 \frac{C_0}{2} T^2 \leq \vareps_0 \int_0^T \int_0^s \Kb(u(\tau),v(\tau)) \dd\tau \dd s \leq C T R(T) + C (R(T))^2
	\]
	which implies $R(T) \geq CT$. In particular, we obtain $\sup_{t\in[0,T]} \Kb(u(t),v(t)) \geq C T^4$. 
	
	$\bullet$ In the non mass-resonance case, we have from \eqref{est-RT} that
	\[
	\Mcal_{\chi_{R(T)}}(T) \leq \Mcal_{\chi_{R(T)}}(0) - \vareps_0\int_0^T \Kb(u(s),v(s))\dd s.
	\]
	which together with \eqref{inf-solu-supe} imply
	\[
	\vareps_0 T \leq \vareps_0\int_0^T \Kb(u(s),v(s)) \dd s \leq |\Mcal_{\chi_{R(T)}}(0)| + |\Mcal_{\chi_{R(T)}}(T)|.
	\]	
	Using $|\Mcal_{\chi_{R(T)}}(0)| \leq C(\kappa,u_0,v_0) R(T)$ and
	\begin{align*}
	|\Mcal_{\chi_{R(T)}}(T)| &\leq \|\nabla \chi_{R(T)}\|_{L^\infty} \left(\|\nabla u(T)\|_{L^2}\|u(T)\|_{L^2} + \|\nabla v(T)\|_{L^2}\|v(T)\|_{L^2}\right) \\
	&\leq C(\kappa, \Mb(u_0,v_0)) R(T) \sqrt{\Kb(u(T),v(T))} \\
	&\leq C(\kappa, \Mb(u_0,v_0)) (R(T))^3,
	\end{align*}
	we obtain
	\[
	\vareps_0 T \leq C R(T) + C (R(T))^3.
	\]
	This shows that $R(T)\geq C T^{1/3}$, hence $\inf_{t\in[0,T]} \Kb(u(t),v(t)) \geq C T^{\frac{4}{3}}$. 
	
	Finally, we prove that \eqref{blow-cond} is fulfilled provided that either $\Hb(u_0,v_0)<0$, or if $\Hb(u_0,v_0) \geq 0$, we assume \eqref{cond-blow-supe-1} and \eqref{cond-blow-supe-2}. 
	
	$\bullet$ The case $\Hb(u_0,v_0) <0$. From \eqref{est-GHK}, we have
	\[
	\Gb(u(t),v(t)) \leq \frac{d+2\alpha}{2} \Hb(u_0,v_0)
	\]
	which proves \eqref{blow-cond} with $\delta =- \frac{d+2\alpha}{2} \Hb(u_0,v_0)>0$.
	
	$\bullet$ The case $\Hb(u_0,v_0) \geq 0$ in which \eqref{cond-blow-supe-1} and \eqref{cond-blow-supe-2} are assumed. By the same argument as in the proof of Proposition \ref{prop-gwp} in the super-critical case, we can prove that
	\begin{align} \label{est-KF-blow}
	\Kb(u(t),v(t)) \left(\Mb(u(t),v(t))\right)^\sigma > \Kb(\varphi,\psi) \left(\Mb(\varphi,\psi)\right)^\sigma, \quad \forall t\in (-T_*,T^*).
	\end{align}
	On the other hand, by \eqref{cond-blow-supe-1}, we can take $\rho \in (0,1)$ such that 
	\[
	\Hb(u_0,v_0) \left(\Mb(u_0,v_0)\right)^\sigma \leq (1-\rho) \Eb_0(\varphi,\psi) \left(\Mb(\varphi,\psi)\right)^\sigma.
	\]
	which together with the conservation of mass and energy yield 
	\begin{align*}
	\Hb(u(t),v(t)) \left(\Mb(u(t),v(t))\right)^\sigma &\leq (1-\rho) \Eb_0(\varphi,\psi) \left(\Mb(\varphi,\psi)\right)^\sigma \\
	&= (1-\rho) \frac{d+2\alpha-4}{2(d+2\alpha)} \Kb(\varphi,\psi) \left(\Mb(\varphi,\psi)\right)^\sigma, \quad \forall t\in (-T_*,T^*).
	\end{align*}
	It follows from \eqref{est-GHK} and \eqref{est-KF-blow} that
	\begin{align*}
	&\Gb(u(t),v(t)) \left(\Mb(u(t),v(t))\right)^\sigma \\
	&\quad \quad \leq \frac{d+2\alpha}{2} \Hb(u(t),v(t)) \left(\Mb(u(t),v(t))\right)^\sigma - \frac{d+2\alpha-4}{4} \Kb(u(t),v(t)) \left(\Mb(u(t),v(t))\right)^\sigma \\
	&\quad \quad \leq (1-\rho)  \frac{d+2\alpha-4}{4} \Kb(\varphi,\psi) \left(\Mb(\varphi,\psi)\right)^\sigma - \frac{d+2\alpha-4}{4} \Kb(\varphi,\psi) \left(\Mb(\varphi,\psi)\right)^\sigma \\
	&\quad \quad= -\rho \frac{d+2\alpha-4}{4} \Kb(\varphi,\psi) \left(\Mb(\varphi,\psi)\right)^\sigma, \quad \forall t\in (-T_*,T^*).
	\end{align*}
	This shows \eqref{blow-cond} with 
	\[
	\delta = \rho \frac{d+2\alpha-4}{4} \Kb(\varphi,\psi) \left(\frac{\Mb(\varphi,\psi)}{\Mb(u_0,v_0)}\right)^\sigma >0.
	\] 
	The proof of Theorem \ref{theo-blow-supe} is now complete.
\end{proof}

\section*{Acknowledgment}
V. D. D. was supported in part by the European Union’s Horizon 2020 Research and Innovation Programme (Grant agreement CORFRONMAT No. 758620, PI: Nicolas Rougerie). A. E. is supported by Nazarbayev University under Faculty Development Competitive Research Grants Program  for 2023-2025 (grant number 20122022FD4121). 

\appendix

\section{Virial identity}
\label{S-appen}
\setcounter{equation}{0}
\begin{lemma} \label{lem-viri-iden}
	Let $u$ is a solution to $i\partial_t u + \beta \Delta u = H$. Then the following identities hold:
	\begin{align*}
	\partial_t |u|^2 &= -2\beta \nabla \cdot \ima (\overline{u} \nabla u) + 2 \ima (\overline{u} H), \\
	\partial_t \ima (\overline{u} \partial_k u) &= \frac{\beta}{2} \partial_k \Delta (|u|^2) - 2\beta \sum_{j=1}^d \partial_j \rea (\partial_j \overline{u} \partial_k u) + 2 \rea (\overline{H} \partial_k u) - \partial_k \rea (H \overline{u}), \quad \forall k=1, \cdots, d.
	\end{align*}
\end{lemma}

\begin{proof}
	We have 
	\begin{align*}
	\partial_t |u|^2 &= 2 \rea (\overline{u} \partial_t u) \\
	&= 2\rea [\overline{u} (i\beta \Delta u - i H)] \\
	&= -2 \beta \ima (\overline{u} \Delta u) + 2 \ima (\overline{u} H) \\
	&= -2\beta \nabla \cdot \ima (\overline{u} \nabla u) + 2\ima (\overline{u} H);
	\end{align*}
	which shows the first identity.
	
	Next we have
	\begin{align*}
	\partial_t \ima (\overline{u} \partial_k u) &= \ima (\overline{\partial_t u} \partial_k u + \overline{u} \partial_k \partial_t u) \\
	&= \ima [(-i\beta \Delta \overline{u} + i \overline{H}) \partial_k u + \overline{u} \partial_k (i\beta \Delta u -i H)] \\
	&=  \rea(-\beta \Delta \overline{u} \partial_k u + \beta \overline{u} \partial_k \Delta u)+ \rea (\overline{H} \partial_k u - \overline{u} \partial_k H). 
	\end{align*}
	As $\partial_k \rea (H \overline{u}) = \rea (\partial_k H \overline{u} + H \partial_k \overline{u})$, we infer that
	\begin{align} \label{viri-iden-prof-1}
	\partial_t \ima (\overline{u} \partial_k u) = \rea(-\beta \Delta \overline{u} \partial_k u + \beta \overline{u} \partial_k \Delta u) + 2 \rea (\overline{H} \partial_k u) - \partial_k \rea (H\overline{u}).
	\end{align}
	On the other hand, we have
	\begin{align*}
	\sum_{j=1}^d \partial_j \rea (\partial_j \overline{u} \partial_k u) &= \sum_{j=1}^d \rea (\partial^2_j \overline{u} \partial_k u + \partial_j \overline{u} \partial^2_{jk} u) \\
	&= \rea (\Delta \overline{u} \partial_k u) +\rea (\nabla \overline{u}\cdot \nabla \partial_k u)
	\end{align*}
	and
	\begin{align*}
	\partial_k \Delta (|u|^2) &= \partial_k \left(\Delta u \overline{u} + 2\nabla u \cdot \nabla \overline{u} + u \Delta \overline{u}\right) \\
	&= \Delta \partial_k u \overline{u} + \Delta u \partial_k \overline{u} + 2 \nabla \partial_k u \cdot \nabla \overline{u} +2\nabla u \cdot \nabla \partial_k \overline{u} + \partial_k u \Delta \overline{u} + u \Delta \partial_k \overline{u} \\
	&= 2 \rea (\Delta \partial_k u \overline{u} + \Delta \overline{u} \partial_k u + 2 \nabla u\cdot \nabla \partial_k \overline{u}).
	\end{align*}
	It follows that
	\[
	\rea (-\beta \Delta \overline{u} \partial_k u + \beta \overline{u} \partial_k \Delta u) = \frac{\beta}{2} \partial_k \Delta (|u|^2) - 2\beta \sum_{j=1}^d \partial_j \rea (\partial_j \overline{u} \partial_k u)
	\]
	which together with \eqref{viri-iden-prof-1} imply the second identity.
\end{proof}

\section*{Acknowledgment}

A. E. is supported by Nazarbayev University under Faculty Development Competitive Research Grants Program  for 2023-2025 (grant number 20122022FD4121).

\vspace{5mm}

\noindent\textbf{Conflict of interest.} The authors declare that they have no conflict of interest.

\noindent\textbf{Data Availability.} There is no data in this paper.

\end{document}